\documentclass[12pt]{article}
\usepackage{amsmath}
\usepackage{amsthm}
\usepackage{amssymb}
\usepackage{xcolor}
\usepackage{caption}
\usepackage{subcaption}
\usepackage{amsfonts}
\usepackage{bm}
\usepackage{stmaryrd}
\usepackage{graphicx}
\usepackage{mathtools}
\usepackage{rotating}
\usepackage{mathrsfs}
\usepackage[section]{placeins}

\newcommand{\half}{\frac{1}{2}}

\include{styles}

\newtheorem{thm}{Theorem}[section]

\newtheorem{lemma}[thm]{Lemma}

\newtheorem{remark}{Remark}[section]

\newtheorem{theorem}{Theorem}[section]

\usepackage{hyperref}

\title{Superconvergence properties of an upwind-biased discontinuous Galerkin method\thanks{The research is supported by the Air Force Office of Scientific Research (AFOSR), Air Force Material Command, USAF, under grant number FA8655-13-1-3017.}} 

\author{Daniel J. Frean\footnotemark[2]
\and Jennifer K. Ryan\footnotemark[2] \footnotemark[3]\thanks{School of Mathematics, University of East Anglia, Norwich, UK}}

\begin{document}
\maketitle
\renewcommand{\thefootnote}{\fnsymbol{footnote}}

\footnotetext[2]{School of Mathematics, University of East Anglia, Norwich, UK}
\footnotetext[3]{Corresponding author (Jennifer.Ryan@uea.ac.uk)}

\begin{abstract}
In this paper we investigate the superconvergence properties of the discontinuous Galerkin method based on the upwind-biased flux for linear time-dependent hyperbolic equations.  We prove that for even-degree polynomials, the method is locally $\mathcal{O}(h^{k+2})$ superconvergent at roots of a linear combination of the left- and right-Radau polynomials.  This linear combination depends on the value of $\theta$ used in the flux.  For odd-degree polynomials, the scheme is superconvergent provided that a proper global initial interpolation can be defined.  We demonstrate numerically that,
for decreasing $\theta$, 
the discretization errors decrease for even polynomials
and grow for odd polynomials.
We prove that the use of Smoothness-Increasing Accuracy-Conserving (SIAC) filters is still able to draw out the superconvergence information and create a globally smooth and superconvergent solution of $\mathcal{O}(h^{2k+1})$ for linear hyperbolic equations.   Lastly, we briefly consider the spectrum of the upwind-biased DG operator and demonstrate that the price paid for the introduction of the parameter $\theta$ is limited to a contribution to the constant attached to the post-processed error term.
\end{abstract}


%
\section{Introduction}
\label{sec: intro}
%

%
Superconvergence is an increasing area of interest in the study of discontinuous Galerkin (DG) methods {\cite{adjerid,AdjeridBaccouch,Baccouch,ChengShu}. In particular, the superconvergence property can be useful in creating a globally higher-order approximation \cite{CockLuskShuSuli,JiXuRyan,Jietal}.  Recently, it has been observed that this property depends on the flux used to construct the 
DG method, specifically for the Lax-Wendroff DG method \cite{qiuLW}.
\\\\
%
In this paper, we study superconvergence properties for the upwind-biased Runge-Kutta discontinuous Galerkin (RKDG) scheme for linear time-dependent hyperbolic equations and the usefulness of Smoothness-Increasing Accuracy-Conserving (SIAC) filtering for drawing out the superconvergent information. 
%
We prove that, for polynomials of even degree $k$, the method is locally $\mathcal{O}(h^{k+2})$ superconvergent at roots of a linear combination of the left- and right-Radau polynomials.  This linear combination depends on the value of $\theta$ used in the flux.  For odd-degree polynomials, the scheme is superconvergent (of the same order) provided that a proper global initial interpolation can be defined.  We demonstrate numerically that for decreasing $\theta$, the discretization errors decrease for even-degree polynomials and grow for odd-degree polynomials. 
In support of the pointwise observations, we prove that the use of SIAC filters is still able to draw out the superconvergence information and create a globally smooth and superconvergent solution of $\mathcal{O}(h^{2k+1}).$  Lastly, we briefly consider the spectrum of the upwind-biased DG operator and show how the flux parameter $\theta$ manifests in expressions for eigenvalues of the amplification matrix. 
We
demonstrate that the coefficients in the expressions for the physically relevant eigenvalues grow with decreasing $\theta$ when the polynomial degree is odd and vice versa for even degree polynomials.
\\\\
%
The DG method is a class of finite element method which uses piecewise polynomials
as both test functions and to form the approximation space. Development of the theory supporting DG solutions to hyperbolic equations was completed by Cockburn, Shu and others in, for example, \cite{CockburnShuII, CockburnShuLDG, ChengShu, mengshuzhangwu, yangshu2012}
while the lecture notes of Cockburn, Karniadakis and Shu \cite{CockburnLectures} offer a thorough detailing of work from the previous millennium. 
%
DG solutions can be highly oscillatory but this imperfection can often be overcome by a post-processing at the final time. Bramble and Schatz \cite{BrambleSchatz} 
utilised information offered by the negative-order Sobolev norm,
which is related to extracting the ``hidden'' superconvergence from special points,
to
develop a local post-processing technique.
In the context of DG approximations for linear hyperbolic equations, this technique was described by Cockburn, Luskin, Shu and S{\"u}li
\cite{CockLuskShuSuli} and later extended and developed by Ryan and others \cite{SRV,Jietal,JiXuRyan,ji2013negative},
who relabelled it as the SIAC filter,
to treat nonlinear equations, non-periodic boundary conditions and non-uniform meshes. 
\\\\
%
%
Superconvergence has long been an area of interest.  Original speculation regarding the superconvergence of DG approximations at 
Radau points was given by Biswas et al.
\cite{Biswas}.  For the upwind flux, these points are
roots of right-Radau polynomials where the approximation exhibits $\mathcal{O}(h^{2k+1})$ superconvergence at the outflow edge and $\mathcal{O}(h^{k+2})$ superconvergence at
roots in the interior of the element \cite{AdjeridAiffaFlaherty,adjerid,AdjeridBaccouch}.  
The pointwise superconvergence proofs include a wide class of equations (elliptic, parabolic and hyperbolic) \cite{Adjerid2007,Adjerid2009,AdjeridBaccouch,AdjeridKlauser,AderidMassey,Baccouch}.   Other superconvergence results include those of Cheng and Shu \cite{ChengShu} and Yang Yang et al. \cite{yangshu2012,yangyangNegative}. 
These latter results
make use of the negative-order norm and
include a description of superconvergence towards a special projection of the solution, a fruitful area of recent research connected in \cite{Guo} to analysis (\cite{NumRes})
via the Fourier approach.
\\\\
%
The numerical flux function
has an intrinsic effect on the acclaimed superconvergent properties. The traditional 
choice
in DG schemes for conservation laws is the upwind flux while the 
majority of theory has been developed
assuming
that the flux is monotone.
Recently, Meng, Shu and Wu \cite{Meng} introduced in the context of DG methods for linear hyperbolic equations a more general flux function: the upwind-biased flux. This choice avoids the requirement of exact knowledge of the eigenstructure of the Jacobian and may reduce numerical dissipation  (yielding a better approximation in smooth regions) 
but
it is made at the cost of the loss of monotonicity. In \cite{Meng}, $L^2$-stability and optimal 
$\mathcal{O}(h^{k+1})$ 
convergence results
were obtained which are comparable with those for the upwind scheme \cite{Richter1988}. 
These results provide the theoretical foundations for our investigations. 
\\\\
%
We follow the procedure outlined in \cite{Baccouch} to obtain \textit{a posteriori} error estimates, defining a new Radau polynomial parametrised by a measure $\theta$ of the amount of information included 
in the flux
from the upwind end. 
In the spectral analysis, which takes a similar approach to~\cite{Guo} and~\cite{NumRes}, we solve an 
eigenvalue problem to compare results with the upwind case.
These approaches
require a global initial projection, as in \cite{Meng}, which is complicated by the nature of the upwind-biased flux. 
%
Analysis of the SIAC filtered error, which is facilitated by a dual analysis in a similar fashion to \cite{ji2013negative},
is largely uncomplicated by the new flux whose effect 
is limited to a contribution to the error constant.
A properly defined initial interpolation would suggest interesting further work.
\\\\
%
The outline of this paper is as follows: In 
$\S$\ref{sec:preliminaries}, we discuss the preliminaries needed to perform both the pointwise and negative-order norm error analysis as well as review
the upwind-biased 
DG
scheme.  In 
$\S$\ref{sec:Pointwise}, we prove that the upwind-biased DG scheme is indeed $\mathcal{O}(h^{k+2})$ pointwise superconvergent at a linear combination of the right- and left-Radau polynomials that depends on $\theta$.  In 
$\S$\ref{sec:SIACerror}, we extend the negative-order norm error 
analyses
to upwind-biased DG schemes and show that this simply affects the constant in the error.  In 
$\S$\ref{sec:Dispersion}, we briefly discuss dispersion analysis related to this scheme.  We end by supporting this analysis with numerical examples in 
$\S$\ref{sec:Numerical} and conclusions in 
$\S$\ref{sec:conclude}. 
~
\section{Preliminaries}
\label{sec:preliminaries}
We begin with some definitions used in the error estimates for DG solutions and SIAC filtering and 
review the construction of the DG scheme.
%
%
%
%
%
%
%
%
\subsection{Tessellation}
Let $\Omega_h$ be a tessellation of a $d$-dimensional bounded domain $\Omega$ into elements $S$ of regular quadrilateral-type shape. 
Denote by $\partial\Omega_h = \bigcup_{S\in\Omega_h}\partial S$ the union of boundary faces $\partial S$ of the elements $S \in \Omega_h$. A face $e$ internal to the domain has associated with it ``left'' and ``right''
cells
$S_L$ and $S_R$ and exterior-pointing normal vectors $\textbf{n}_L = (n^L_1,\dots,n^L_d)$ and $\textbf{n}_R = (n^R_1,\dots,n^R_d)$ respectively as
in~\cite{JiXuRyan}. Given a function $v$ defined on neighbouring
cells
$S_L$ and $S_R$ which share a face $e$, we refer to its restriction 
to $e$
in $S_L$
by writing $v^L := (v|_{S_L})|_e$ and similarly for $v^R$.

For clarity, we detail the discretization of 
$\Omega = [-1,1]^2$ into $N_x\cdot N_y$ rectangular cells. Elements take the form $S = I_i \times J_j$, where $I_i$ and $J_j$ are intervals given by $$I_i = \left(x_{i-\half},x_{i+\half}\right),
~~1\leq i\leq N_x,
 \quad\textrm{and }\quad J_j = \left(y_{j-\half},y_{j+\half}\right),
 ~~1\leq j\leq N_y,
 $$
where $-1=x_{\half} < x_{\frac{3}{2}} < \dots <x_{N_x+\half}=1$ and $-1=y_{\half} < y_{\frac{3}{2}} < \dots <y_{N_y+\half}=1.$
Denoting by $h_{x,i} = x_{i+\half} - x_{i-\half}$
and by $h_{y,j} = y_{j+\half} - y_{j-\half}$,
we define $$h = \max\left\lbrace\max_{1\leq i\leq N_x}h_{x,i},\max_{1\leq j\leq N_y}h_{y,j}\right\rbrace$$ and require regularity: $h_{x,i},h_{y,j} \geq ch,~c>0.$
%
On an element $S = I_i\times J_j$, evaluation of functions $v$ at cell boundary points of the form $(x_{i+\half},y)$, for example, is written as $v^L
|_{x=x_{i+\half}}
 = v\left(x_{i+\half}^-,y\right).$
%
%
%
%
\subsection{Basis polynomials}
We use as basis functions the orthogonal Legendre polynomials $P_n(\xi)$, which can be defined by the Rodrigues formula
\begin{equation}\label{eqn:Rodrigues}
P_n(\xi) = \frac{1}{2^nn!}\frac{d^n}{d\xi^n}\left((\xi^2-1)^n\right),\quad -1\leq\xi\leq 1,\quad n\geq 0.
\end{equation}
Denote by $\delta_{nm}$ the Kronecker-delta function. Useful properties include
\begin{subequations}
\vspace{-8pt}
\begin{equation}\label{eqn:LegendreOrthog}
\int_{-1}^1P_n(\xi)P_m(\xi)~\textrm{d}\xi = \frac{2}{2n+1}\delta_{nm};
\end{equation}
\vspace{-15pt}
\begin{equation}\label{eqn:LegendreProperties}
P_n(1) = 1;\quad P_n(-1) = (-1)^n; \quad \frac{d}{d\xi}P_n(1) = \half n (n+1);
\end{equation}
\vspace{-15pt}
\begin{equation}\label{eqn:Legendrederiv}
\frac{d}{d\xi}P_{n+1}(\xi) = (2n+1)P_n(\xi) + (2n-3)P_{n-2}(\xi) + (2n-7)P_{n-4}(\xi) + \dots.
\end{equation}
\end{subequations}
Using the Legendre polynomials, we define the right- and left-Radau polynomials as
\begin{equation*}
R^+_{k+1}(\xi) = P_{k+1}(\xi) - P_k(\xi),\quad\quad R^-_{k+1}(\xi) = P_{k+1}(\xi) + P_k(\xi)
\end{equation*}
respectively. 
Note that 
the roots 
$\xi^+_1 < \xi^+_2 < \dots < \xi^+_{k+1} = 1$
of $R^+_{k+1}(\xi)$ and the roots 
$-1 = \xi^-_1 < \xi^-_2 < \dots < \xi^-_{k+1}$
of $R^-_{k+1}(\xi)$ are real, distinct and lie in the interval $[-1,1]$. 
%
%
%
%
\subsection{Function spaces}
Due to the tensor-product nature of the post-processing kernel, we will require the function space $\mathcal{Q}^k(S)$ of tensor-product polynomials of degree at most $k$ in each variable. Thus we 
define
the following finite element spaces:
\begin{subequations}
\begin{eqnarray*}
\qquad 
V^k_h &=& \lbrace\phi \in \mathcal{L}^2(\Omega) :\phi|_S\in\mathcal{Q}^k(S),\forall S\in\Omega_h\rbrace;\label{eqn:Vkh}
\\
\qquad
\Sigma^k_h &=& 
\lbrace\mathbf{\eta} = (\eta_1, \dots,\eta_d)^T\in\left(\mathcal{L}^2(\Omega)\right)^d:
\eta_l\in\mathcal{Q}^k(S), 
l=\leq d;\forall S\in\Omega_h\rbrace,
\end{eqnarray*} 
\end{subequations}
where $\mathcal{L}^2(\Omega)$ is the space of square-integrable functions on $\Omega$. Nevertheless,
it has been observed (\cite{ryanshuatkins}) that the filter also works for the standard polynomial space $\mathscr{P}^k(S)$. Note that
if $\Omega$ is
one-dimensional,
these function spaces $\mathcal{Q}^k(S)$ and $\mathscr{P}^k(S)$ agree.
%
%
%
%
\subsection{Operators on the function spaces}
We list
some
standard notations. 
Denote by $\mathbb{P}_hv$ and 
$\Pi_h \boldsymbol{p}$ the 
$\mathcal{L}^2$-projections of scalar- and vector-valued functions $v$ and $\boldsymbol{p}$.
%
The inner-product over $\Omega$ of two
scalar- or vector-valued
functions is 
given as
\vspace{-8pt}
\begin{equation}
\vspace{-8pt}
(w,v)_{\Omega} = \sum_S\int_Swv~\textrm{d}S;\quad (\boldsymbol{p},\boldsymbol{q})_{\Omega} = \sum_S\int_S\boldsymbol{p}\cdot\boldsymbol{q}~\textrm{d}S.
\end{equation}
%
The $\mathcal{L}^2$-norm on the domain $\Omega$ and on the boundary $\partial\Omega$ is defined as
\vspace{-8pt}
\begin{equation}
\vspace{-8pt}
\lVert\eta\rVert_{\Omega} = \left(\int_{\Omega}\eta^2~\textrm{d}\boldsymbol{x}\right)^{1/2};\quad \lVert\eta\rVert_{\partial\Omega} = \left(\int_{\partial\Omega}\eta^2~\textrm{d}s\right)^{1/2}.
\end{equation}
The $\ell$-norm and semi-norm of the Sobolev space $H^{\ell}(\Omega)$ are defined respectively as
\vspace{-8pt}
\begin{equation}
\vspace{-8pt}
\lVert\eta\rVert_{\ell,\Omega} = \left(\sum_{|\alpha|\leq\ell}\lVert D^{\alpha}\eta\rVert_{\Omega}^2\right)^{1/2};
\quad\quad 
|\eta|_{\ell,\Omega} = \sum_{|\alpha|\leq\ell}\lVert D^{\alpha}\eta\rVert_{\infty,\Omega},\quad \ell > 0,
\end{equation}
where $\alpha$ is a $d$-dimensional multi-index of order $|\alpha|$ and where $D^{\alpha}$ denotes multi-dimensional partial derivatives. The definitions for
these
norms for vector-valued functions are analogous to the scalar case.
In 
$\S$\ref{sec:SIACerror}, we
utilise the negative-order
norm
\vspace{-8pt}
\begin{equation}
\vspace{-8pt}
\lVert\eta\rVert_{-\ell,\Omega} = \sup_{\Phi\in\mathcal{C}^{\infty}_0(\Omega)}\frac{(\eta,\Phi)_{\Omega}}{\lVert\Phi\rVert_{\ell,\Omega}}
\end{equation}
as a means of obtaining $\mathcal{L}^2$-error estimates for the filtered solution. Note that
\vspace{-8pt}
\begin{equation}
\vspace{-8pt}
\lVert\eta\rVert_{-\ell,\Omega} \leq \lVert\eta\rVert_{\Omega},
\qquad \forall \ell \geq 1.
\end{equation} 
The negative-order norm can be used to detect oscillations of a function~(\cite{CockLuskShuSuli}) and is connected to the SIAC filter which smooths oscillations in the error.
Finally, the difference quotients $\partial_{h,j}v$ are given by
\vspace{-8pt}
\begin{equation}
\vspace{-8pt}
\partial_{h,j}v(\boldsymbol{x}) = \frac{1}{h}\left[v(\boldsymbol{x}+\half h\textbf{e}_j)~-~v(\boldsymbol{x}-\half h\textbf{e}_j)\right],
\end{equation}
where $\textbf{e}_j$ is the $j^{\textrm{th}}$ component unit normal vector. For any multi-index $\alpha = (\alpha_1,\dots,\alpha_d)$, we define the $\alpha^{\textrm{th}}$-order difference quotient by
\vspace{-8pt}
\begin{equation}
\vspace{-8pt}
\partial_{h,j}^{\alpha}v(\boldsymbol{x}) = \left(\partial_{h,1}^{\alpha_1}\cdot\cdot\cdot\partial_{h,d}^{\alpha_d}\right)v(\boldsymbol{x}).
\end{equation}
\subsection{The convolution kernel}
We detail the component parts of the
SIAC
filter as defined in
\cite{Jietal}.
A B-spline $\psi^{(\ell)}$ of order $\ell$ is defined recursively by
\vspace{-8pt}
\begin{equation*}
\vspace{-8pt}
\psi^{(\ell)} = \chi_{[-\half,\half]};\quad \psi^{(\ell)} = \psi^{(\ell-1)}\star\chi_{[-\half,\half]},~~\ell\geq 2,
\end{equation*} 
where $\chi_{[-\half,\half]}$ is the characteristic function on the interval $[-\half,\half]$ and where the operator $\star$ denotes convolution:
\vspace{-8pt}
\begin{equation*}\vspace{-8pt}
f(x) \star g(x) = \int_{\mathbb{R}}f(x-y)g(y)~\textrm{d}y.
\end{equation*}
For a multi-index $\alpha$
and
given a point $\boldsymbol{x} = (x_1,\dots,x_d) \in \mathbb{R}^d$, we define 
\vspace{-8pt}\begin{equation*}\vspace{-8pt}
\psi^{(\alpha)}(\boldsymbol{x}) = \psi^{(\alpha_1)}(x_1)\cdot\cdot\cdot\psi^{(\alpha_d)}(x_d);
\qquad
\psi^{(\ell)}(\boldsymbol{x}) = \psi^{(\ell)}(x_1)\cdot\cdot\cdot\psi^{(\ell)}(x_d).
\end{equation*}
In this way, we construct a convolution kernel
\vspace{-5pt}
\begin{equation}\vspace{-8pt}
\mathcal{K}_h^{(r+1,\ell)}(\boldsymbol{x})  
= \sum_{\gamma\in\mathbb{Z}^d}\boldsymbol{c}_{\gamma}^{r+1,\ell}\psi^{(\ell)}(\boldsymbol{x}-\gamma)
\end{equation} 
which comprises a linear combination of $r+1$ B-splines $\psi^{(\ell)} \in \mathcal{C}^{\ell-2}$ of order $\ell$ such that $\mathcal{K}_h^{(r+1,\ell)}$ has compact support and reproduces (by convolution) polynomials of degree strictly less than $r$. Typically, $r = 2k$ and $\ell = k+1$, where $k$ is the degree of the polynomial basis.
The coefficients $\boldsymbol{c}_{\gamma}$ are tensor products of the coefficients $c_{\gamma}$ found by requiring the reproduction of polynomials property $\mathcal{K}_h^{(r+1,\ell)}\star x^p = x^p,~p<r,$ in the one-dimensional case. It is important to note that derivatives of a convolution with this kernel may be written in terms of difference quotients:
\begin{equation}
D^{\alpha}\left(\psi_h^{(\beta)}\star v\right) = \psi_h^{(\beta-\alpha)}\star\partial_h^{\alpha}v,\quad \beta_i \geq \alpha_i,
\end{equation}
where $\psi_h^{(\beta)}(x) = \psi_h^{(\beta/h)}/h^d$. Further properties of the kernel may be found in~\cite{JiXuRyan}. 
\\
By convolving the approximation with the kernel, we obtain
the SIAC filtered solution \vspace{-10pt}
\begin{equation}\vspace{-5pt}
u_h^{\star}(\bar{x}, t) \coloneqq \mathcal{K}^{(r+1,\ell)}_h(\bar{x})\star u_h(\bar{x},t),
\end{equation}
which
displays increased accuracy and
reduced
oscillations in the error.
The results in this paper treat only the symmetric kernel where the nodes $\gamma$ are uniformly spaced. Extension to the one-sided filter given in~\cite{Jietal} and \cite{SRV} is a straight-forward task.  
~\newline

%
\subsection{
DG
discretization of the 
linear hyperbolic conservation law
}
%
Let $\boldsymbol{x} = (x_1,\dots,x_d)\in\Omega \subset \mathbb{R}^d$ and
consider the linear hyperbolic system 
\vspace{-8pt}
\begin{eqnarray}\label{eqn:linearhyperbolicsystem}\vspace{-8pt}
u_t + \sum_{i=1}^df_i(u)_{x_{i}} &=& 0,\quad\quad\quad (\boldsymbol{x},t)\in\Omega\times (0,T],
\end{eqnarray}
for the conserved quantity $u(\boldsymbol{x},t)$
with
a linear flux function $f_i(u) = a_iu
$ with real constant coefficients $a_{i} \geq 0$. We adopt throughout this paper the assumptions of a smooth initial condition and of periodicity in the boundary conditions: $$u(\boldsymbol{x},0) = u_0(\boldsymbol{x})\in\mathcal{C}^{\infty}\left(\Omega\right);\quad\quad u(\boldsymbol{x},0) = u(\boldsymbol{x},T).$$
For much of the error analyses, 
one requires only
$u_0(\boldsymbol{x})\in\mathcal{H}^{k+1}\left(\Omega\right)$ but for the proof of Theorem~\ref{Thm:pointwise}, 
the stronger condition is required
for
a Maclaurin series. 
%
%
%
%

Given a tessellation $\Omega_h$,
the method is facilitated by multiplying equation~(\ref{eqn:linearhyperbolicsystem}) by a test function $v$ and integrating by parts over an arbitrary element $S\in\Omega_h$ to obtain
\begin{align}\label{exactDG method}\vspace{-15pt}
\int_Su_tv~\textrm{d}S - \sum_{i=1}^d\int_Sf_i(u)v_{x_{i}}~\textrm{d}S &+ \sum_{i=1}^d\int_{\partial S} f_i(u)v~\textrm{d}s  ~=~ 0.
\vspace{-8pt}\end{align}
We 
adopt 
the piecewise-polynomial approximation space 
$V_h^k$ defined in equation~(\ref{eqn:Vkh}).
Note that functions $v \in V_h^k$ are allowed to be discontinuous across element boundaries. This is the distinguishing feature of DG schemes amongst finite element methods.
By replacing in equation~(\ref{exactDG method}) the solution $u(\boldsymbol{x},t)$ by a numerical approximation $u_h(\boldsymbol{x},t)$ such that $u_h(\cdot,t)\in V_h^k$, we obtain the discontinuous Galerkin method: find, for any $v \in V_h^k$ and for all elements $S$, the unique function $u_h(\cdot,t) \in V^k_h$ which satisfies \vspace{-10pt}
\begin{align}\label{DG method}\vspace{-10pt}
\int_S(u_h)_tv~\textrm{d}S - \sum_{i=1}^d\int_Sf_i(u_h)v_{x_{i}}~\textrm{d}S &+ \sum_{i=1}^d\int_{\partial S} \hat{f}_i(u_h)n_iv~\textrm{d}s  ~=~ 0,
\end{align}
where $\hat{f}$ is a single-valued numerical flux function used to enforce weak continuity at the cell interfaces and where $\textbf{n} = (n_1,\dots,n_d)$ is the unit normal vector to the domain of integration.
The initial condition $u_h(\boldsymbol{x},0) \in V^k_h$ is usually taken to be the $\mathcal{L}^2$-projection $\mathbb{P}_hu_0$, although the analysis in $\S$\ref{sec:Pointwise} favours a function which interpolates $u(\boldsymbol{x},0)$ at the superconvergent points.
Summing equation~(\ref{DG method}) over the elements $S$, we get a compact expression for the global scheme:
\begin{equation*}
\left((u_h)_t, v\right)_{\Omega_h} + B(u_h; v) = 0,
\end{equation*}
where we define for future use
\begin{equation}
 B(u_h; v) \coloneqq -\sum_{i=1}^d\left(f_i(u_h),v_{x_{i}}\right)_{\Omega_h} + \sum_{i=1}^d\left(\hat{f}_i(u_h)n_i,v\right)_{\partial\Omega_h}.
\end{equation}
%
%
%
%
In order to ensure stability of the scheme~(\ref{DG method}), it remains to define the numerical flux functions $\hat{f}$ featured in the cell boundary terms. In general, $\hat{f}\left(u_h^L,u_h^R\right)$ depends on values of the numerical solution from both sides of the cell interface. Traditionally (\cite{CockburnShuII}), this function is chosen to be a so-called monotone flux,
which satisfies Lipschitz continuity, consistency ($\hat{f}(u,u) = f(u)$) and monotonicity ($\hat{f}(\uparrow,\downarrow)$). 
For our test equation~(\ref{eqn:linearhyperbolicsystem}), where the linear flux
determines a single wind direction, the usual choice in the literature is to satisfy the upwinding condition. In this paper, where $\hat{f}(u_h) = a\widehat{u}_h$, we choose instead the upwind-biased flux 
\begin{equation}\label{eqn:upwindbiased flux}
\widehat{u}_h = \theta u_h^L + (1-\theta)u_h^R,\quad\quad \theta = (\theta_1,\dots,\theta_d),\quad\half<\theta_i\leq 1,
\end{equation}
defined here for periodic boundary conditions, which was recently described in the context of DG methods by Meng et al. in~\cite{Meng}. 
More information is taken from the left than from the right of cell boundaries and, when $\theta_i = 1$ for each $i$, the upwind-biased flux reduces to the purely upwind flux $u^L_h$. We do not allow $\theta_i = \half$, which gives a central flux, since then the scheme becomes unstable. 
%
For clarity, we particularise the evaluation of this flux at cell boundary points in two dimensions:
\begin{eqnarray*}
\widehat{u}_h|_{x=x_{i+\half}}~ &=&~ \theta_1 u_h\left(x_{i+\half}^-,y\right) ~+~ (1-\theta_1)u_h\left(x_{i+\half}^+,y\right)\quad\textrm{at } (x_{i+\half},y),
\\
\widehat{u}_h|_{y=y_{j+\half}}~ &=&~ \theta_2 u_h\left(x,y_{j+\half}^-\right) ~+~ (1-\theta_2)u_h\left(x,y_{j+\half}^+\right)\quad\textrm{at } (x,y_{j+\half}^+).
\end{eqnarray*}
Choosing, over the upwinding principle, the upwind-biased flux offers several rewards~(\cite{Meng}): a possibly reduced numerical viscosity and easier construction, for example. However, 
for values $\theta < 1$,
we sacrifice the established property of monotonicity.
In this paper, we consider in terms of superconvergence the severity of this loss.
%
%
%
%
%
\section{Superconvergent pointwise error estimate}
\label{sec:Pointwise}%

In this section, we demonstrate that when the flux in the DG scheme is chosen to be the upwind-biased flux, the leading order term in the error is proportional to a sum, dependent upon $\theta$,
of left- and right-Radau polynomials. The main result, Theorem~\ref{Thm:pointwise}, is an extension of the observation, for example of Adjerid, Baccouch and others (\cite{adjerid,AdjeridBaccouch}), that the superconvergent points for the purely upwind DG scheme are generated by roots of right-Radau polynomials. To this end, we define a ``special'' Radau polynomial 
\begin{equation}\label{eqn:specialRadau}
R^{\star}_{k+1}(\xi) \coloneqq \theta R^+_{k+1}(\xi) + (-1)^k(1-\theta)R^-_{k+1}(\xi),
\qquad\xi\in[-1,1],\quad\half<\theta\leq1.
\end{equation}
We will show that roots of $R_{k+1}^{\star}(\xi)$, which change with the value of $\theta$, generate superconvergent points on the order of $h^{k+2}$ for the upwind-biased scheme. 
%
%
%
Interestingly, it turns out that for odd polynomial degree $k$, one of these ``superconvergent points'' lies outside the element $[-1,1]$ when $\theta < 1$. 
In
the following Lemma, 
we describe 
the roots of $R^{\star}_{k+1}(\xi)$. For this argument, 
we consider the polynomials $P_n(\xi)$ arising from the Rodrigues formula~(\ref{eqn:Rodrigues}) 
and then extend their domain of definition to $[-M,M]$ for some fixed, sufficiently large $M>0$.
Of course, any root that we find to be outside $[-1,1]$ will not directly manifest as a superconvergent point of the DG solution. 

\begin{lemma}
\label{lmm:oddvseven}
Let $k \in \mathbb{N}$ and consider 
\begin{equation*}
R^{\star}_{k+1}(\xi) = \theta R^+_{k+1}(\xi) + (-1)^k(1-\theta)R^-_{k+1}(\xi), \quad\quad\xi\in\mathbb{R},\qquad
\half<\theta\leq1.
\end{equation*} 
When $k$ is even, all $k+1$ roots
of $R^{\star}_{k+1}(\xi)$
lie in the interval $[-1,1]$. When $k$ is odd, exactly one root
is greater than $1$ while all other roots lie in the interval $[-1,1]$. 
\end{lemma}
\begin{proof}
We split the proof into two cases,
writing the special Radau polynomial as 
\begin{eqnarray}
R^{\star}_{k+1} = \begin{cases}
P_{k+1} - (2\theta-1)P_k,\quad &\textrm{ when } k \textrm{ is even},
\\
(2\theta-1)P_{k+1} - P_k,\quad &\textrm{ when } k \textrm{ is odd}.
\end{cases}
\end{eqnarray}
\\
Suppose that $k$ is even and let  $\half < \theta \leq 1$. It is clear that there is no root of $R^{\star}_{k+1}(\xi)$ greater than $1$ since $R^{\star}_{k+1}(1) = 2(1-\theta) > 0$ and, by 
equation~(\ref{eqn:Legendrederiv}), we have, for $\xi > 1$,
\begin{eqnarray*}
\frac{d}{d\xi}R^{\star}_{k+1}(\xi) &=& \left[(2k+1)P_k - (2\theta-1)(2k-1)P_{k-1}\right]
\\
&&+
\left[(2k-3)P_{k-2} - (2\theta-1)(2k-5)P_{k-3}\right]
+\dots
+P_0
.
\end{eqnarray*}
Similarly, there is no root $\xi < -1$ since $R^{\star}_{k+1}(-1) = -2\theta < 0$ and $\frac{d}{d\xi}R^{\star}_{k+1}(\xi) > 0$.
\\
Suppose instead that the polynomial degree $k$ is odd. To see that $R^{\star}_{k+1}(\xi)$ has a root $\xi = c > 1,$ note that while $R^{\star}_{k+1}(1) = -2(1-\theta) \leq 0$,
we have 
$\lim_{\xi\rightarrow\infty}R^{\star}_{k+1}(\xi) = +\infty.$
Appealing to the Intermediate Value Theorem yields 
the required root.
Furthermore, there is only one such root since, for all $\xi>1$, 
property~(\ref{eqn:Legendrederiv}) 
gives
\begin{eqnarray*}
\frac{d}{d\xi}R^{\star}_{k+1}(\xi) 
&\geq& (2\theta-1)[(2k+1)(P_k - P_{k-1}) + (2k-3)(P_{k-2}-P_{k-3}) 
\\
&&
+ \dots + 3(P_1-P_0)]
>0.
\end{eqnarray*}
A similar argument shows that there are no roots $\xi < -1$: while $R^{\star}_{k+1}(-1) = 2\theta > 0$, we have that $\frac{d}{d\xi}R^{\star}_{k+1}(\xi) < 0$ for all $\xi < -1$. 
\qquad\end{proof}
\begin{remark}\normalfont
Recall that when $\theta = 1$, one of the superconvergent points is the downwind end $\xi^{\star}_{k+1} = \xi^+_{k+1} = 1$. When $k$ is even, 
the roots of $R^{\star}_{k+1}(\xi)$
shift to the left with decreasing values of $\theta
$. On the other hand, when $k$ is odd, the points shift to the right and $\xi^{\star}_{k+1} > 1$. For example, when $k = 1$, the roots of $R^{\star}_{k+1}(\xi)$ are given by
$$\xi^{\star}_{1,2} = \frac{1\mp\sqrt{1-3\theta+3\theta^2}}{3(2\theta-1)}.$$
\begin{center}
\begin{table} \footnotesize\caption{Approximations to roots $\xi^{\star}_j,~j=1,\dots,k+1,$ of $R^{\star}_{k+1}(\xi) \textrm{ when } \theta=1,
\frac{3}{4}
$.}
\label{tab:roots of Rstar}
 \begin{tabular}{|c|| l l l l l| l l l l l||} 
 \hline
$k$ & \multicolumn{5}{c |}{$\xi^{\star}_j = \xi^{+}_j$ when $\theta = 1$} & \multicolumn{5}{c ||}{$\xi^{\star}_j$ when $\theta = 
\frac{3}{4}$} \\ [0.5ex] 
 \hline\hline
$1$ & 
$-\frac{1}{3}$&$1$&&& & $-0.21$&$1.54$&&& \\ 
 \hline
$2$ & $-0.68$&$0.28$&$1$&& & $-0.72$&$0.16$&$0.86$&& \\
 \hline
$3$ & $-0.82$&$-0.18$&$0.57$&$1$& & $-0.80$&$-0.11$&$0.69$&$1.36$& \\
 \hline
$4$ & $-0.88$&$-0.44$&$0.16$&$0.72$&$1$ & $-0.89$&$-0.48$&$0.09$&$0.62$&$0.93$ \\ [1ex] 
 \hline
\end{tabular}
\end{table}
\end{center}
Shown in Table~\ref{tab:roots of Rstar} are approximate values of the roots of $R^{\star}_{k+1}(\xi)$ for two values of $\theta$. 
Observe that
for $k = 1$ and $k = 3$, one of the roots is indeed greater than $1$.
\end{remark}\noindent

Following the lines of~\cite{AdjeridBaccouch}, 
we interpolate the initial condition at roots of $R^{\star}_{k+1}(\xi)$, where $k$ is even. Lemma~\ref{lmm:oddvseven} dictates that we cannot obtain in the same way a $k^{\textrm{th}}$ degree polynomial interpolating $u$ at all $k+1$ roots of $R^{\star}_{k+1}(\xi)$ when $k$ is odd.
Instead, one can define a global projection similar to~\cite{Meng} but we leave this as further work.
%
%
\begin{lemma}\label{lmm:interpolating special}
Let $k > 0$ be an even integer and suppose that $u \in \mathcal{C}^{k+1}\left([0,h]\right)$. 
Let $\xi^{\star}_j\in[-1,1],~ j=1,\dots,k+1$, be the roots of 
$
R_{k+1}^{\star}(\xi)$ 
as defined by~(\ref{eqn:specialRadau}).
Consider the $k^{\textrm{th}}$ degree Lagrange polynomial
$$\pi^{\star}u(x) = \sum_{n=1}^{k+1}L_n(x),\quad\quad
L_n(x) = u(x_n^{\star})\prod_{\substack{j=1\\ j\neq n}}^{k+1}\frac{x-x^{\star}_{j}}{x^{\star}_n-x^{\star}_{j}},
\quad\quad
 x\in[0,h]
 ,
 $$
interpolating $u$ at the (distinct) roots 
$x_{j}^{\star} = \frac{h}{2}(\xi_j^{\star} + 1)$ of the shifted special Radau polynomial $R^{\star}_{k+1}(x)$ on $[0,h]$.
Then the interpolation error satisfies
\begin{align}
u(x(\xi)) - \pi^{\star}u(x(\xi)) =  h^{k+1}c_{k+1}R^{\star}_{k+1}(\xi) + \sum_{\ell=k+2}^{\infty}Q_{\ell}(\xi)h^{\ell},
\end{align}
where $Q_{\ell}(\xi)$ is a polynomial of degree at most $\ell$.
\end{lemma}
%
%
%
{\em Proof}.
The standard Lagrangian interpolation theory yields, for $x \in [0,h]$, an $s = s(x) \in (0,h)$ such that 
\begin{equation*}
e(x) = u(x) - \pi^{\star}u(x)
 = 
 \frac{1}{(k+1)!}u^{(k+1)}(s(x))\prod^{k+1}_{j=1}\left(x - x^{\star}_{j}\right).
\end{equation*}
By the linear mapping $x = \frac{h}{2}(1 + \xi)$, we obtain
\begin{equation}\label{eqn:eofxofxi}
e\left(x(\xi)\right) =  \frac{h^{k+1}u^{(k+1)}\left(s(x(\xi))\right)}{2^{k+1}(k+1)!}\prod^{k+1}_{j=1}\left(\xi - \xi^{\star}_{j}\right).
\end{equation}
The Maclaurin series of $u^{(k+1)}\left(s(x(\xi))\right)$ with respect to $h$ gives the leading order term in the error so that equation~(\ref{eqn:eofxofxi}) becomes
\begin{equation}
e\left(x(\xi)\right) =  h^{k+1}\frac{u^{(k+1)}(0)}{2^{k+1}(k+1)!}\prod^{k+1}_{j=1}\left(\xi - \xi^{\star}_{j}\right) +  \sum_{m=1}^{\infty} Q_m(\xi)h^{m+k+1},
\end{equation}
where $Q_{m}(\xi)$ comprises the product of $R^{\star}_{k+1}(\xi)$ and a polynomial of degree $m$ in $\xi$:
$$
Q_{m}(\xi) = \frac{\frac{d^{m}}{dh^{m}}\frac{\partial^{k+1} u\left(s(x(\xi))\right)}{\partial x^{k+1}}|_{h=0}}{2^{k+1}(k+1)! m!}\prod^{k+1}_{j=1}\left(\xi - \xi^{\star}_{j}\right).
\qquad\qedhere
$$
\\
The interpolatory polynomials described in Lemma~\ref{lmm:interpolating special}  
are used as initial conditions in the proof of Theorem~\ref{Thm:pointwise}.
However, 
numerical results in $\S$\ref{sec:Numerical} confirm that, in general, there are only $k$ superconvergent points in each element when $k$ is odd. Thus we proceed 
to
the main result in this section assuming an even polynomial degree. 
%
%
%
\begin{theorem}\label{Thm:pointwise}
Let $k > 0$ be an even integer. Consider the approximate solution $u_h$ to
equation~(\ref{eqn:linearhyperbolicsystem}) with $d = 1$ obtained by a DG scheme~(\ref{DG method}) using $k^{\textrm{th}}$ order basis functions, a uniform mesh and the upwind-biased flux $\widehat{u}_h$. Let the numerical initial condition be the interpolating polynomial $\pi^{\star}u(x,0)$ described in Lemma~\ref{lmm:interpolating special}. 
\\
Let $\xi = \frac{2}{h}x - \frac{1}{h}(x_{j+\half}+x_{j-\half})$ be the scaling between the cell $I_j$ and the canonical element $[-1,1]$. Then the error $e = u - u_h$ satisfies
\begin{eqnarray}
e(\xi,h,t) ~=~\sum_{\ell=k+1}^{\infty}Q_{\ell}(\xi,t)h^{\ell}, \quad\quad Q_{\ell}(\cdot,t)\in\mathscr{P}^{\ell}([-1,1]),
\end{eqnarray}
with 
\begin{eqnarray*}
Q_{k+1}(\xi,t)
&=&
 c_{k+1}(u,h,k,t)
R^{\star}_{k+1}(\xi) 
\\
&=& c_{k+1}(u,h,k,t)
\left(\theta R^{+}_{k+1}(\xi) + (1-\theta)R^{-}_{k+1}(\xi)\right).
\end{eqnarray*}
\end{theorem}
{\em Proof}.
Without loss of generality, 
assume the
tessellation $\Omega_h$
to
comprise
a single element
$[0, h]$.
To facilitate the analysis, subtract the approximating scheme~(\ref{DG method}) from equation~(\ref{exactDG method}) to obtain a DG orthogonality condition for the error $e$:
\begin{equation}
\int_0^h e_tv~\textrm{d}x - \int_0^h aev_{x}~\textrm{d}x + a\widehat{e}v|^{x=h}_{x=0} ~=~ 0,\label{DG method, error}
\end{equation}
where $\widehat{e} = u - \widehat{u}_h.$ 
The flux terms in equation~(\ref{DG method, error}) can be evaluated using the periodicity of the boundary conditions as follows:
\begin{align}\label{ehat expansion}
\widehat{e}|_{x=0} = 
\widehat{e}|_{x=h} &= \theta(u - u_h^-)|_{x=h} + 
(1-\theta)(u - u_h^+)|_{x=h}
\\
&= \theta e|_{x=h} + 
(1-\theta)e|_{x=0}\nonumber
\end{align} 
Substitution of
the cell boundary evaluations~(\ref{ehat expansion})
into
equation~(\ref{DG method, error})
yields for all $v \in V^k_h$,
after a scaling to the canonical element $[-1,1]$,
\begin{align}\label{DG method, error,scaled}
&\frac{h}{2}\int_{-1}^1e_tv~\textrm{d}\xi - \int_{-1}^1aev_{\xi}~\textrm{d}\xi 
+  a\left(\theta e|_{\xi=1} + (1-\theta)e|_{\xi=-1}\right)\left(v(1) - v(-1)\right)  ~=~ 0.
\end{align}
Next, we reformat
equation~(\ref{DG method, error,scaled}) as a scheme for the leading order terms of the error.
\\
%
%
\textit{Step One}:\quad
The DG solution within an element is clearly analytic as a function of $h$. Since we assume an initial condition of class $\mathcal{C}^{\infty}$, the advecting exact solution is also smooth 
so the (local) DG solution is analytic
in $\xi.$ Hence we can expand the (local) error $e = u - u_h$, which is analytic, as a Maclaurin series with respect to $h$:
\begin{equation}\label{Eqn:33}
e(\xi,h,t) = \sum_{\ell=0}^{\infty}Q_{\ell}(\xi,t)h^{\ell},
\end{equation}
where $Q_{\ell}(\cdot,t) = \sum_{m=0}^{\ell}b_mP_m(\xi)$ is a polynomial of degree at most $\ell$. 
\\
Next, substitute the expansion~(\ref{Eqn:33}) into the scaled scheme~(\ref{DG method, error,scaled}) for the error and collect terms in powers of $h$ to obtain the following two equations:
\begin{subequations}
\begin{align}\label{Eqn:10a}
&- \int_{-1}^1Q_0v_{\xi}~\textrm{d}\xi 
+ \left(\theta Q_0(1,t) + (1-\theta)Q_0(-1,t)\right)\left(v(1) - v(-1)\right)  ~=~ 0;
\end{align}
\begin{align}\label{Eqn:10b}
&\frac{1}{2}\int_{-1}^1(Q_{\ell-1})_tv~\textrm{d}\xi - \int_{-1}^1aQ_{\ell}v_{\xi}~\textrm{d}\xi \\
&
+ a\left(\theta Q_{\ell}(1,t) + (1-\theta)Q_{\ell}(-1,t)\right)\left(v(1) - v(-1)\right)  ~=~ 0,
\qquad
\ell \geq 1
.\nonumber
\end{align}
\end{subequations}
Since $Q_0(\xi,t) = Q_0(t) = \theta Q_0(t) + (1 - \theta)Q_0(t)$,
the Fundamental Theorem of Calculus 
immediately satisfies equation~(\ref{Eqn:10a}).
It is from equation~(\ref{Eqn:10b}), by inductively testing against functions $v \in V^k_h$, that the rest of the argument is extracted. 
\\
%
%
\textit{Step Two}:\quad 
Substitute $\ell = 1$ in equation~(\ref{Eqn:10b}) and choose $v = 1$ 
to obtain
\begin{equation*}
\int_{-1}^1(Q_{0})_t~\textrm{d}\xi~=~(Q_{0})_t\int_{-1}^1~\textrm{d}\xi~=~0.
\end{equation*}
Thus we must have $(Q_{0})_t(\xi,t) = 0$. 
Any $(k+1)$-node interpolating
initial condition
$\pi u_0(x)$ leads to the first $k+1$ coefficients in the expansion~(\ref{Eqn:33}) vanishing initially:
\begin{equation*}
Q_{\ell}(\xi, 0)~=~0,\quad \ell = 0,\dots,k.
\end{equation*}
In particular, since $(Q_{0})_t(\xi,t) = 0$, we have
$Q_0(\xi,t)=0$
for all $t$. This last observation forms the base step for an induction on $k$ in equation~(\ref{Eqn:10b}) with the hypothesis $Q_{\ell}(\xi,t) = 0,~\ell=1,\dots,k-1.$
\\
%
%
\textit{Step Three}:\quad
To show that $Q_{k}(\xi,t) = 0$, consecutively substitute $\ell = k$ and $\ell = k+1$ in equation~(\ref{Eqn:10b}) and choose $v = \xi$ and $v = 1$ respectively to obtain in turn
\begin{subequations}
\vspace{-10pt}
\begin{align}\label{l is k v is xi}
& - \int_{-1}^1Q_k ~\textrm{d}\xi + 2\left(\theta Q_k(1,t) + (1-\theta)Q_k(-1,t)\right)  ~=~ 0,
\end{align}
\vspace{-20pt}
\begin{equation}\label{Qk sub t is zero}
\int_{-1}^1 (Q_k)_t~\textrm{d}\xi~=~0.
\end{equation} 
\end{subequations}
After differentiating equation~(\ref{l is k v is xi}) with respect to $t$,  equation~(\ref{Qk sub t is zero}) yields
\begin{align}\label{eqn:toint}
&\frac{d}{dt}\left(\theta Q_k(1,t) + (1-\theta)Q_k(-1,t)\right)  ~=~ 0.
\end{align}
Since $Q_k(\xi,0) = 0$, the
integral in time of equation~\eqref{eqn:toint} leaves, for any $t$,
\begin{align}\label{flux in 10b}
&\theta Q_k(1,t) + (1-\theta)Q_k(-1,t) ~=~ 0
\end{align}
The terms
in
equation~(\ref{flux in 10b}) feature in
equation~(\ref{Eqn:10b}) when $l = k$
thus we obtain 
\begin{equation*}
\int_{-1}^1Q_kv_{\xi}~\textrm{d}\xi~=~0,\quad v\in V^k_h. 
\end{equation*}
Hence $Q_k$ is orthogonal to all $v\in\mathscr{P}^{k-1}$ and, if we write $Q_k(\cdot,t)$ and $v(\xi)$ as sums of Legendre polynomials $P_{\ell}(\xi)$, orthogonality properties~(\ref{eqn:LegendreOrthog}) yield
\begin{equation}\label{Eqn:Qk expansion}
Q_k(\xi,t)~=~\sum_{\ell=0}^kb_{\ell}(t)P_{\ell}(\xi)~=~b_k(t)P_k(\xi).
\end{equation}
The expansion~(\ref{Eqn:Qk expansion}) must satisfy the flux condition
(\ref{flux in 10b}) when it follows
 that
\begin{align*}
&\theta b_k P_k(1) + (1-\theta)b_k P_k(-1)  ~=~ 0
\end{align*}
so, by Legendre properties~(\ref{eqn:LegendreProperties}),
$b_k = 0$. This completes the induction; 
$$Q_{\ell}(\xi,t) = 0,\quad\ell=0,1,\dots,k.$$
\textit{Step Four}:\quad
We consider the term $Q_{k+1}$ following the same process as before. That is, consecutively substitute $\ell = k+1$ and $\ell = k+2$ in equation~(\ref{Eqn:10b}) and choose $v = \xi$ and $v = 1$ respectively to obtain in turn
\begin{subequations}
\vspace{-10pt}
\begin{align}\label{l=k+1,psi=xi}
&\int_{-1}^1Q_{k+1}~\textrm{d}\xi = 2\left(\theta Q_{k+1}(1,t) + (1-\theta)Q_{k+1}(-1,t)\right),
\end{align}
\vspace{-20pt}
\begin{equation}\label{l=k+2,psi1}
\int_{-1}^1(Q_{k+1})_t~\textrm{d}\xi~=~0.
\end{equation}
\end{subequations}
Next, differentiate equation~(\ref{l=k+1,psi=xi}), equate to zero using equation~(\ref{l=k+2,psi1}) and apply the Fundamental Theorem of Calculus to obtain
\begin{equation}\label{eqn:Qk+1 periodic condition}
\theta Q_{k+1}(1,t) + (1-\theta)Q_{k+1}(-1,t)~=~\theta Q_{k+1}(1,0) + (1-\theta)Q_{k+1}(-1,0).
\end{equation}
To see that the right-hand side of equation~(\ref{eqn:Qk+1 periodic condition}) vanishes when $k$ is even, 
recall that 
the leading order term in the interpolation error $u_0(x) - \pi^{\star}u_0(x)$ 
satisfies
\begin{equation*}
Q_{k+1}(\xi,0) = c_{k+1}R^{\star}_{k+1}(\xi)
\end{equation*}
then note that, 
irrespective of the value of $k$, the following equates to zero: 
\begin{eqnarray*}
\theta R^{\star}_{k+1}(1) + (1-\theta)R^{\star}_{k+1}(-1) &=& \theta(1-\theta)\left[
(-1)^k
R^{-}_{k+1}(1) + R^{+}_{k+1}(-1)\right]
\\
&=& \theta(1-\theta)\left[2
(-1)^k 
+ 2(-1)^{k+1}\right].
\end{eqnarray*}
It follows from equation~(\ref{eqn:Qk+1 periodic condition}) that we also have, for even $k$ and for all $t\geq0$, 
\begin{equation}\label{flux Qk+1}
\theta Q_{k+1}(1,t) + (1-\theta)Q_{k+1}(-1,t)~=~0.
\end{equation}
Thus the flux terms in equation~(\ref{Eqn:10b}) with $\ell = k+1$ vanish which leaves
\begin{equation*}\label{Qk+1 orthogonal}
\int_{-1}^1Q_{k+1}v_{\xi}~\textrm{d}\xi = 0,\quad\quad v \in V^k_h.
\end{equation*}
Hence $Q_{k+1}$ is orthogonal to all $v\in\mathscr{P}^{k-1}$ and we can write, at any given $t > 0$,
\begin{equation}\label{Qk+1}
Q_{k+1}(\xi,t)~=~b_{k+1}(t)P_{k+1}(\xi) + b_{k}(t)P_{k}(\xi).
\end{equation}
If we require of the expansion~(\ref{Qk+1}) the conditions~(\ref{flux Qk+1}) then we must satisfy
\begin{align*}
\theta\left[b_{k+1} + b_{k}\right] + (1 - \theta)\left[(-1)^{k+1}b_{k+1} + (-1)^kb_{k}\right]~=~0,
\end{align*}
From which
it follows that 
\begin{align*}
b_{k} = \begin{cases} -(2\theta - 1) b_{k+1}, &\quad \textrm{when } k \textrm{ is even}
\\
-\frac{1}{2\theta - 1}b_{k+1}, &\quad \textrm{when } k \textrm{ is odd.}
\end{cases}
\end{align*}
In the case of an even polynomial degree $k$, we have
\begin{eqnarray*}
Q_{k+1}(\xi,t) ~&=&~ 
b_{k+1}P_{k+1}(\xi)  + (1 - 2\theta)b_{k+1}P_k(\xi) 
\\
&=&
b_{k+1}\left[\theta R^+_{k+1}(\xi) + (1-\theta)R^-_{k+1}(\xi)\right],
\end{eqnarray*}
where $b_{k+1}$ depends on $t$.
\qquad\qedhere
%
%
%
%
\begin{remark}\label{rem:end pointwise}\normalfont
When $u_h(x,0) = \pi^{\star}u_0(x)$ interpolates $u_0(x)$ at the roots of $R^{\star}_{k+1}(x)$, the coefficient of the term on the order of $h^{k+1}$ in the series for the initial error satisfies
\begin{equation}\label{eqn:inherent condition}
\theta Q_{k+1}(1,0) + (1-\theta)Q_{k+1}(-1,0) = 0.
\end{equation}
In the proof of Theorem~\ref{Thm:pointwise}, this relation was extended to $t > 0$.  
For odd $k$, 
optimal construction of $\pi^{\star}u_0(x)$ is hindered
by the root of $R^{\star}_{k+1}(\xi)$ which lies outside $[-1,1]$.
If, however, we were able to satisfy equation~(\ref{eqn:inherent condition}) for odd $k$ then 
we would 
get
\begin{eqnarray*}
Q_{k+1}(\xi,t) ~&=&~ \frac{1}{2\theta-1}\left[
(2\theta-1)b_{k+1}P_{k+1}(\xi)  -  b_{k+1}P_k(\xi)\right]
\\
&=&\frac{1}{2\theta - 1}
b_{k+1}\left[\theta R^+_{k+1}(\xi) - (1-\theta)R^-_{k+1}(\xi)\right]
=\frac{b_{k+1}}{2\theta - 1}R_{k+1}^{\star}(\xi).
\end{eqnarray*}
\end{remark}\noindent
For simplicity of exposition, 
Theorem~\ref{Thm:pointwise} 
was restricted to the one-dimensional case but extension of the results to multiple dimensions, when the approximation space consists of piecewise continuous tensor polynomials, is reasonably straightforward. 
\\
In the next section, we will compliment
our
pointwise observations
with an analysis of the global error. When considering the global error, we are able to extract the full $\mathcal{O}\left(h^{2k+1}\right)$ superconvergence rate using 
SIAC
filtering regardless of
the parity of $k$.
%
%

%
\section{SIAC filtered error estimation}
\label{sec:SIACerror}
The hidden local accuracy of the DG solution, discussed in 
$\S3$,
may be extracted to a global measure by applying the  smoothness-increasing accuracy-conserving (SIAC) filter introduced by \cite{SRV}. In this section, we show that
$\mathcal{O}(h^{2k+1})$
superconvergent accuracy 
in the negative-order norm, as is observed (\cite{Jietal}) for the upwind flux, still occurs when the upwind-biased DG method is used to solve linear hyperbolic conservation laws. To begin, 
we observe that an error bound in the $\mathcal{L}^2$-norm follows from a negative-order norm error estimate.
Let 
$$u_h^{\star} = \mathcal{K}^{(2k+1,k+1)}_h\star u_h$$ 
be the DG solution to equation~(\ref{eqn:linearhyperbolicsystem}) post-processed by convolution kernel at the final time.   Denote by $e_h = u - u_h$ the usual DG error and 
consider the $\mathcal{L}^2$-norm of the error $e_h^{\star}\coloneqq u - u_h^{\star}$ associated with the filtered solution:
\begin{equation}\label{eqn:L2 of estar}
\lVert u - u_h^{\star}\rVert_0 = \lVert u - \mathcal{K}_h\star u\rVert_0 + \lVert \mathcal{K}_h\star u - u_h^{\star}\rVert_0.
\end{equation}
The first term on the right-hand side of~(\ref{eqn:L2 of estar}) is bounded by $Ch^{r+1}$ from the integral form of Taylor's theorem and from the reproduction of polynomials property of the convolution (Lemma $5.1$, \cite{Jietal}). Thus we need only consider the second term for which
\begin{equation}
\lVert\mathcal{K}_h\star u - u_h^{\star}\rVert_0 = \lVert\mathcal{K}_h\star e_h\rVert_0 \leq \sum_{|\alpha|\leq\ell}\lVert D^{\alpha}(\mathcal{K}_h\star e_h)\rVert_0
\leq \sum_{|\alpha|\leq \ell}\lVert\tilde{\mathcal{K}}_h\rVert_1\lVert \partial_h^{\alpha}e_h\rVert_{-\ell}\label{eqn:36}
\end{equation}
by kernel properties of the $\alpha^{\textrm{th}}$ derivative $D^{\alpha}$, the kernel's relation to the divided difference $\partial^{\alpha}$ and by Young's inequality for convolutions. The tilde on $\tilde{\mathcal{K}}_h$ in inequality~(\ref{eqn:36}) signals that the kernel uses $\psi^{(\ell-\alpha)}$, which is a result of the property $D^{\alpha}\psi^{(\ell)} = \partial_h^{\alpha}\psi^{(\ell-\alpha)}$. 

Note that $\lVert\tilde{\mathcal{K}}_h\rVert_{\ell} = \sum_{i=0}^r|c_{i}|$ is just the sum of the kernel coefficients so we only need to show that $\lVert\partial^{\alpha}_he_h\rVert_{-\ell} \leq Ch^{2k+1}$. Furthermore, the formulation of the DG scheme for the solution is similar to that for the divided differences   
\begin{equation}
\lVert\partial^{\alpha}_h(u - u_h)\rVert_{-\ell,\Omega} \leq C\lVert\partial_h^{\alpha}u_0\rVert_{\ell,\Omega}h^{2k+m},\quad m\in\left\lbrace 0,
1/2,1\right\rbrace.
\end{equation}
\cite{CockLuskShuSuli}.
This allows us to only have to consider the negative-order norm of the solution itself; superconvergent accuracy in the negative-order norm gives superconvergent accuracy in the $\mathcal{L}^2$-norm for the post-processed solution. The following result provides the required negative-order norm error estimate.
\begin{remark}
Notice that the superconvergent points for the upwind-biased scheme, as described in the one-dimensional case in Lemma~\ref{lmm:interpolating special}, 
change with the value of $\theta$. However, the global superconvergence in the negative-order norm occurs regardless of the value of $\theta$. Furthermore, the proof of the following result does not differ between odd and even polynomial degrees. 
\end{remark}
%
%
%
%
\begin{theorem}\label{Thm:Filtered Thm}
Let $u_h$ be the numerical solution to
equation~(\ref{eqn:linearhyperbolicsystem})
with smooth initial condition obtained via a
DG
scheme~(\ref{DG method}) with upwind-biased flux.
Then 
\begin{equation}
\left\lVert \partial^{\alpha}_h\left(u - u_h\right)(T)\right\rVert_{-\ell,\Omega}~\leq~C(u_0,\theta,T)h^{2k+1},\quad\quad\alpha < \ell.
\end{equation}
\end{theorem}
%
%
{\em Proof}.  The proof is quite similar to \cite{CockLuskShuSuli}. In the following, we point out the differences.
For simplicity, we consider the case when $\alpha = 0$. The case for $\alpha > 0$ is similar (\cite{CockLuskShuSuli,yangyangNegative}).
In order to extract information about the error at the final time, we work with the dual equation: find a continuous and analytic 
$\phi(\boldsymbol{x},t)$ such that
\vspace{-15pt}\begin{equation}\label{Eqn:Dual equation}\vspace{-15pt}
\phi_t + \sum_{i=1}^da_i\phi_{x_i} = 0;\quad\quad\phi(\boldsymbol{x},T) = \Phi(\boldsymbol{x}),\quad\quad (\boldsymbol{x},t) \in \Omega\times[0,T).
\end{equation}
Thus we can estimate the term appearing in the definition of the negative-order norm as
\vspace{-10pt}
\begin{equation}\label{Theta123}\vspace{-10pt}
(u - u_h,\Phi)_{\Omega}(T) = \Theta_1 + \Theta_2 + \Theta_3
\end{equation}
where
the terms on the right-hand side of equation~(\ref{Theta123}) can be estimated individually:
\vspace{-15pt}\begin{align}\vspace{-15pt}
&\Theta_1~=~(u - u_h,\phi)_{\Omega}(0);\nonumber
\\
&\Theta_2 ~=~ - \int_0^T\left[\left((u_h)_t, \phi-\chi\right)_{\Omega} + B(u_h,\phi-\chi)\right]~\textrm{d}t;\nonumber
\\
&\Theta_3 ~=~ - \int_0^T\left[\left(u_h, \phi_t\right)_{\Omega} - B(u_h,\phi)\right]~\textrm{d}t.\nonumber
\end{align}
%
%
%
%
We summarize the estimates for each of these terms below.
\subsubsection*{Estimating the term $\Theta_1$}
The bounding of this projection term is no different to that in~\cite{Jietal} and is given by 
\vspace{-10pt}\begin{equation}\vspace{-10pt}
|\Theta_1| 
\leq C_1h^{2k+2}\lVert u_0\rVert_{k+1}\lVert\phi(0)\rVert_{k+1}.
\end{equation}
\subsubsection*{Estimating the term $\Theta_3$}
The estimate for $\Theta_3$ also does not differ from previous analysis and hence we have
\begin{eqnarray}
|\Theta_3| &=& \left|\int_0^T\left(u_h, \phi_t\right)_{\Omega} - B(u_h;\phi)~\textrm{d}t\right|~=~0.
\end{eqnarray}
%
%
%
\subsubsection*{Estimating the term $\Theta_2$}
In contrast to $\Theta_1$ and $\Theta_3$, estimation of this term is affected by the choice of flux parameter $\theta$ and so we provide more detail. 

Let $\mathbb{P}_h\phi$ be the projection onto the approximation space. Then since the projection error $\phi - \mathbb{P}_h\phi$ is orthogonal to the approximation space, it holds that
$$\left((u_h)_t, \phi-\mathbb{P}_h\phi\right)_{\Omega} = 0.$$ 
This leaves us only needing to bound one term: $$\Theta_2 = - \int_0^TB(u_h;\phi-\mathbb{P}_h\phi)~\textrm{d}t.$$
Denote by $[\![\phi]\!]_{\partial S} = \phi^R_{\partial S}\textbf{n}_R  + \phi^L_{\partial S}\textbf{n}_L$ the jump in $\phi$. When we choose $\widehat{u}_h$ to be the upwind-biased flux, the integrand in the above equation becomes
\begin{align}
B(u_h;\phi-\mathbb{P}_h\phi) = \sum_S\int_{\partial S}[\![ u_h]\!]_{\partial S}\left((\phi - \mathbb{P}_h\phi^R_{\partial S}\textbf{n}_R) - (1-\theta)[\![\phi-\mathbb{P}_h\phi]\!]_{\partial S}\right)~\textrm{d}s.
\end{align}
Let $\theta_{\min} = \min\{\theta_1,\dots,\theta_d\}$. Then 
\begin{align}
|\Theta_2| &= \left|\int_0^TB(u_h;\phi-\mathbb{P}_h\phi)~\textrm{d}t\right|
\\
&\leq C_a\left(\frac{1}{h}\int_0^T\lVert  [\![ u_h]\!]
 \rVert^2_{\Omega}\textrm{d}t\right)^{1/2}
\left(\int_0^T\lVert \phi-\mathbb{P}_h\phi^R
\textbf{n}_R\rVert^2_{\partial\Omega}\textrm{d}t\right)^{1/2}
\nonumber\\
&+(1-\theta_{\min})C_a\left(\frac{1}{h}\int_0^T\lVert  [\![ u_h]\!]\rVert^2_{\Omega}\textrm{d}t\right)^{1/2}
\left(\frac{1}{h}\int_0^T\lVert  [\![ \phi-\mathbb{P}_h\phi]\!]\rVert^2_{\partial\Omega}\textrm{d}t\right)^{1/2}\nonumber
\\
&\leq C_a\left(C_u\frac{1}{h}h^{2(k+1)}\right)^{1/2}\left[C_qh^{k+\half}\lVert\phi\rVert_{k+1} 
+C_p(1-\theta_{\min})h^{k+\half}\lVert\phi\rVert_{k+1}\right]\nonumber
\\
&=~C_aC_u\left[C_q + (1-\theta_{\min})C_p\right]h^{2k+1}\lVert\phi\rVert_{k+1}(T) 
\nonumber
\\
&=~C_2h^{2k+1}\lVert\Phi\rVert_{k+1},\nonumber
\end{align}
where the constant $C_2 = C^{+}_{2} + (1-\theta_{\min})C^-_{2}$ depends on $\theta$. 
\\
Combining the estimates and using the periodicity of the boundary conditions, we conclude with a bound on the numerator in the definition of the negative-order norm:
\begin{eqnarray}
(u - u_h,\Phi)_{\Omega}(T) &= \Theta_1 + \Theta_2 + \Theta_3 \label{eqn:ineq}
\\
&\leq C_1h^{2k+2}\lVert u_0\rVert_{k+1}\lVert\phi(0)\rVert_{k+1}
+ C_2h^{2k+1}\lVert\Phi\rVert_{k+1}.
\nonumber\qquad\qedhere
\end{eqnarray}
%
%
%
%
\begin{remark}
The
penalty for
using the new flux
is limited to a contribution to the constant attached to the order term in the negative-order norm error estimate
and
we can extract the same global order of accuracy, $\mathcal{O}\left(h^{2k+1}\right)$, for any polynomial degree $k$.
This is in contrast to the changing local behaviour seen in the pointwise analysis in $\S$\ref{sec:Pointwise}. 
\end{remark}\noindent

To complete our description of how the superconvergent properties of the upwind-biased DG solution change with the value of $\theta$,
we conduct a short 
analysis of the eigenvalues of the amplification matrix associated with the spatial discretization.
%
%
%
%

%
\section{Dispersion analysis}
\label{sec:Dispersion}
A further 
analysis of
DG
methods which has proved fruitful in recent years 
follows a Fourier approach (\cite{Guo,NumRes,ChengLi2014}).
%
The choice of initial condition and basis functions can be crucial in obtaining optimal results. Recent work that demonstrates the importance of this choice includes~\cite{ChengLi2014,yangshu2012} and \cite{cao2014superconvergence}. 
In what follows, we analyse the eigenvalues, which are independent of the choice of basis.
\\
Consider the local DG solution 
$$u_h(x(\xi),t)|_{I_j} = \sum_{\ell=0}^{k}u^{(\ell)}_j(t)\phi^{\ell}_j(\xi),
\qquad
\phi^{\ell}_j(\xi)\in V^k_h,
$$
to equation~(\ref{eqn:linearhyperbolicsystem}) with $d = 1$, periodic boundary conditions and a uniform mesh.
Denote by $\boldsymbol{u}_j$ the vector whose entries are the coefficients $u^{(\ell)}_j(t),~\ell=0,\dots,k$.
Then the DG scheme with upwind-biased flux can be written as the following semi-discrete system of ODEs:
\begin{equation}\label{eqn:semi-discrete scheme}
\frac{d}{dt}\boldsymbol{u}_j ~=~ \frac{a}{h}\left[(A_1+\theta A_2)\boldsymbol{u}_j + \theta B\boldsymbol{u}_{j-1} + (1-\theta)C\boldsymbol{u}_{j+1}\right],
\end{equation}
where $A = A_1+\theta A_2, B$ and $C$ are $(k+1)\times(k+1)$ matrices. 
%
%
%
%
Note that the term $C\boldsymbol{u}_{j+1}$ from the right-neighbour cell is a new contribution when $\theta < 1$.
%
%
As in~\cite{NumRes} and~\cite{Guo}, the coefficient vectors can be transformed to Fourier space via the assumption 
\begin{equation}\label{Fourier assumption}
\boldsymbol{u}_j(t) = e^{i\omega x_j}\hat{\boldsymbol{u}}_{\omega}(t),
\end{equation}
where $\omega$ is the wave number, $i = \sqrt{-1}$ and $x_j$
is the element center, to obtain a global coefficient vector $\hat{\boldsymbol{u}}_{\omega}$. Substitution of the ansatz~(\ref{Fourier assumption}) into the scheme~(\ref{eqn:semi-discrete scheme}) gives
\begin{equation}\label{ODE for uHat}
\frac{d}{dt}\hat{\boldsymbol{u}}_{\omega} = aG(\omega, h)\hat{\boldsymbol{u}}_{\omega},
\end{equation}
where 
\begin{equation}
G(\omega, h) = \frac{1}{h}\left(A + \theta Be^{-i\omega h} + (1-\theta)Ce^{i\omega h} \right)
\end{equation}
is called the amplification matrix.
If $G$ is diagonalisable then it has a full set of eigenvalues $\lambda_1, \dots, \lambda_{k+1}$ and corresponding eigenvectors $\Lambda_1, \dots, \Lambda_{k+1}$.
%
%
\\
Using Mathematica to computationally perform an asymptotic analysis on $\zeta = \omega h = 0$, we can obtain the following sets of eigenvalues $\lambda_j$ of the amplification matrix
$G$:
\begin{eqnarray*}
k = 0:~&&\lambda_1 = -i\omega~-~\half\left(2\theta-1\right)\omega^2h~+~\mathcal{O}\left(h^2\right);
\\\nonumber\\
k = 1:~&&\begin{cases}
\lambda_1 = -i\omega-\frac{1}{72}\frac{1}{2\theta-1}\omega^4h^3
-\frac{i}{270}\frac{1 + 6\theta-6\theta^2}{(1-2\theta)^2}\omega^5h^4 + \mathcal{O}\left(h^5\right),\\\\
\lambda_2 =-\frac{6(2\theta-1)}{h} + 3i\omega + (2\theta - 1)\omega^2h + \mathcal{O}\left(h^2\right);
\end{cases}
\\\nonumber\\
k = 2:~&&\begin{cases}
\lambda_1 =-i\omega-\frac{2\theta - 1}{7200}\omega^6h^5 
+ \frac{i}{300}\left[\theta^2-\theta+\frac{1}{14}\right]\omega^7h^6 +\mathcal{O}\left(h^7\right),\\\\
\lambda_{2,3} =~-\frac{3(2\theta-1)}{h}\pm i\sqrt{51 + 36\theta - 36\theta^2}\omega+\mathcal{O}\left(h\right);
\end{cases}\label{G1 taylor eigenvalues}
\\\nonumber\\
k = 3:~&&
\lambda_1 =-i\omega-\frac{3.125\times 10^{-4}}{441(2\theta - 1)}\omega^8h^7~
\\
&&- \frac{1.25\times 10^{-3}}{27783}\frac{19-48\theta+28\theta^2}{(1-2\theta)^2}\omega^9h^8+\mathcal{O}(h^9).
\\\nonumber
\end{eqnarray*}
%
%
%
For each value of $k$, the eigenvalue $\lambda_1$ has physical relevance, approximating $-i\omega$ with dispersion error on the order of $h^{2k+1}$ and dissipation error on the order of $h^{2k+2}$. This is consistent with the previous findings of \cite{Ainsworth,Sherwin,HuAtkins02,HeLiQiu,NumRes} and \cite{Guo}.
The coefficient of the leading order real term of the physically relevant eigenvalues $\lambda_1$ is negative. While for even $k$ this coefficient vanishes in the limit $\theta\rightarrow\half$, for odd $k$, due to the factor $(2\theta-1)^{-1}$, it grows without bound with reducing values of $\theta$. This (blow-up) behaviour is amplified in the coefficient of $h^{2k+2}$. Note that such differences between odd and even $k$ do not manifest when $\theta = 1$ since then $2\theta -1 = 1$.
\\
The remaining eigenvalues are non-physically relevant but have negative real part on the order of $\frac{1}{h}$.
Thus 
the corresponding eigenvectors in the solution
are 
damped over time,
which 
occurs slowly
for lower values of $\theta$.
For the case $k = 3$,
the findings are consistent with the other cases but
the algebra involved in the computation
becomes prohibitively substantial and the need to evaluate components numerically makes it particularly difficult to obtain tidy 
expressions for the coefficients.
%
%
%
\begin{remark}
While the eigenvalues are independent of the choice of basis functions, one must make an appropriate choice of interpolating initial condition and basis functions in order to extract superconvergent accuracy in the eigenvectors. If one uses a Lagrange-Radau basis on roots of $R^+_{k+1}(\xi),$ the appropriate choice when $\theta = 1$, in the case $k = 1$ the physically relevant eigenvector satisfies
\begin{equation}
C_1\Lambda_1 - \hat{\boldsymbol{u}}(0)~=~\left[\begin{array}{c}
-\frac{1-\theta}{18(1-2\theta)}\omega^2 h^2~+~ \frac{11-11\theta + 2\theta^2}{324(1-2\theta)^2}i\omega^3 h^3~+~\mathcal{O}(h^4)
\\
\frac{1-\theta}{6(1-2\theta)}\omega^2 h^2~+~ \frac{1-25\theta + 22\theta^2}{108(1-2\theta)^2}i\omega^3 h^3~+~\mathcal{O}(h^4)

\end{array}\right],
\end{equation}
while similarly the non-physically relevant eigenvector satisfies
\begin{equation}
C_2\Lambda_2 ~=~\left[\begin{array}{c}
\frac{1-\theta}{18(1-2\theta)}\omega^2 h^2~-~ \frac{11-11\theta + 2\theta^2}{324(1-2\theta)^2}i\omega^3 h^3~+~\mathcal{O}(h^4)
\\
-\frac{1-\theta}{6(1-2\theta)}\omega^2 h^2~-~ \frac{1-25\theta + 22\theta^2}{108(1-2\theta)^2}i\omega^3 h^3~+~\mathcal{O}(h^4)

\end{array}\right].
\end{equation}
The leading order terms vanish only when $\theta = 1$ when the interpolation points coincide with the superconvergent points of the scheme. Numerical results suggest that for $k = 2$, when we have $k+1$ roots of $R^{\star}_{k+1}(\xi)$, we are able to obtain the optimal $\mathcal{O}\left(h^{2k+1}\right)$ accuracy by using $u_h(x,0) = \pi^{\star}_{k+1}u_0(x)$. We leave as further work construction of an appropriate basis in the case of odd polynomial degree.
\end{remark}\noindent
%
%
%
%

%
\newpage
\section{Numerical Experiments}
\label{sec:Numerical}
%

We present a numerical discussion for the test equation
\vspace{-10pt}\begin{eqnarray}\vspace{-10pt}
u_t + u_x = 0,\quad\quad (x,t)\in[0,2\pi]\times(0,T],
\label{eqn:testeqn}
\\
u(x,0) = \sin(x),\quad u(0,T) = u(2\pi,T)\nonumber
\end{eqnarray}
solved by the DG scheme with upwind-biased flux paired with the Strong Stability Preserving (SSP) Runge Kutta SSP$(3,3)$ timestepping method described in \cite{gottlieb2011strong}.  The {\sf cfl} is taken so that spatial errors dominate.

Figures~\ref{fig:DiscErrk2}
 and \ref{fig:DiscErrk3} show the DG discretization errors on a grid of $N = 10$ elements for various values of $\theta$ and for polynomial degrees $k = 2,3.$ Marked by the red crosses are the theoretical superconvergent points which are roots of $R^{\star}_{k+1}(\xi)$ and which change with the value of $\theta\in\left(\half,1\right]$. The error curves cross the zero axis near these roots. In~\cite{adjerid}, Adjerid et al. commented that the intersection points align more closely as $k$ increases and we observe that here too. 

For cases with even polynomial degree,
we observe $k+1$ superconvergent points while for the odd cases, in general, the error curves cross the zero axis only $k$ times.
Furthermore, as the value of $\theta$ reduces, we see an overall reduction in the magnitude of the errors for even $k$. On the other hand, when $k$ is odd the magnitude of the errors in general increases for smaller values of $\theta$.
%
Tables 
\ref{tab:k=2} and \ref{tab:k=3}  
illustrate the $\mathcal{O}(h^{k+1})$ accuracy of the DG solution in the $\mathcal{L}^2$- and $\mathcal{L}^{\infty}$-norms. After post-processing by the SIAC filter, we observe the $\mathcal{O}\left(h^{2k+1}\right)$ accuracy in the $\mathcal{L}^2$-norm described in $\S$\ref{sec:SIACerror} and we also see $\mathcal{O}\left(h^{2k+1}\right)$ accuracy in the $\mathcal{L}^{\infty}$-norm. 
For odd $k$, convergence to the expected orders is slower for lower values of $\theta$ but is eventually achieved. Furthermore, if one compares the same degrees of mesh refinement for decreasing values of $\theta$, one observes increasing errors for odd $k$ and reducing errors for even $k$. For the post-processed solution, this is due in large part to the contribution of $\theta$ to the constant attached to the order term in the error estimate of Theorem~\ref{Thm:Filtered Thm}.

The highly oscillatory nature of the DG solution, indicating the existence of the hidden superconvergent points, can be seen in Figures~\ref{fig:Errorsk2} and \ref{fig:Errorsk3} alongside the post-processed solutions which have increased smoothness and improved accuracy. The reduced numerical viscocity enforced by the upwind-biased flux is evident when comparing plots for $\theta = 1$ and $\theta = 0.55$.

%
%
%
%
\vspace{-10pt}
\begin{table}
\begin{center}\footnotesize
\caption{$\mathcal{L}^{2}-$ and $\mathcal{L}^{\infty}$-norms of errors before and after post-processing for case $k = 2$.}
\vspace{-10pt}
\tabcolsep=0.06cm
\begin{tabular}{|r||c|c|c|c||c|c|c|c|}
\hline
 &\multicolumn{4}{c|}{$\mathscr{P}^2$: Before filter}&\multicolumn{4}{|c|}{$\mathscr{P}^2$: After filter}\\
\cline{2-9}
mesh &$\mathcal{L}^{2}$ error&order&$\mathcal{L}^{\infty}$error&order&$\mathcal{L}^{2}$ error&order&$\mathcal{L}^{\infty}$error&order\\
\hline
\multicolumn{9}{|c|}{$\theta = 1$}\\
\hline\hline
$ 10$&$8.59$E$-04$&-&$3.02$E$-03$&$-$&$1.43$E$-04$&$-$&$2.04$E$-04$&$-$\\
$ 20$&$1.06$E$-04$&$3.00$&$3.66$E$-03$&$3.04$&$2.52$E$-06$&$5.83$&$3.85$E$-06$&$5.83$\\
$ 40$&$1.33$E$-05$&$2.99$&$4.62$E$-05$&$2.98$&$4.46$E$-08$&$5.81$&$6.34$E$-08$&$5.82$\\

\hline
\multicolumn{9}{|c|}{$\theta = 0.85$}\\
\hline\hline
$ 10$&$7.35$E$-04$&-&$2.61$E$-03$&$-$&$1.41$E$-04$&$-$&$2.01$E$-04$&$-$\\
$ 20$&$9.03$E$-05$&$3.02$&$3.10$E$-04$&$3.07$&$2.44$E$-06$&$5.86$&$3.47$E$-06$&$5.86$\\
$ 40$&$1.12$E$-05$&$3.00$&$3.85$E$-05$&$3.00$&$4.19$E$-08$&$5.86$&$5.95$E$-08$&$5.86$\\

\hline

\multicolumn{9}{|c|}{$\theta = 0.55$}\\
\hline\hline
$ 10$&$5.66$E$-04$&-&$1.46$E$-03$&$-$&$1.36$E$-03$&$-$&$1.93$E$-04$&$-$\\
$ 20$&$6.97$E$-05$&$3.01$&$1.86$E$-04$&$2.97$&$2.26$E$-06$&$5.91$&$3.20$E$-06$&$5.91$\\
$ 40$&$8.70$E$-06$&$3.00$&$2.31$E$-05$&$3.00$&$3.63$E$-08$&$5.95$&$5.15$E$-08$&$5.96$\\

\hline
\end{tabular}\label{tab:k=2}
\vspace{-10pt}
\end{center}
\end{table}
\begin{table}\footnotesize
\begin{center}
\caption{$\mathcal{L}^{2}-$ and $\mathcal{L}^{\infty}$-norms of errors before and after post-processing for case $k = 3$.}
\vspace{-10pt}
\tabcolsep=0.06cm
\begin{tabular}{|r||c|c|c|c||c|c|c|c|}
\hline
 &\multicolumn{4}{c|}{$\mathscr{P}^3$: Before filter}&\multicolumn{4}{|c|}{$\mathscr{P}^3$: After filter}\\
\cline{2-9}
mesh &$\mathcal{L}^{2}$ error&order&$\mathcal{L}^{\infty}$error&order&$\mathcal{L}^{2}$ error&order&$\mathcal{L}^{\infty}$error&order\\
\hline
\multicolumn{9}{|c|}{$\theta = 1$}\\
\hline\hline
$ 10$&$2.35$E$-04$&-&$1.91$E$-04$&$-$&$1.61$E$-05$&$-$&$2.28$E$-05$&$-$\\
$ 20$&$1.30$E$-05$&$4.16$&$1.06$E$-05$&$4.16$&$6.97$E$-08$&$7.86$&$9.81$E$-08$&$7.86$\\
$ 40$&$8.67$E$-07$&$3.91$&$7.33$E$-07$&$3.86$&$3.34$E$-10$&$7.69$&$4.72$E$-10$&$7.69$\\

\hline
\multicolumn{9}{|c|}{$\theta = 0.85$}\\
\hline\hline
$ 10$&$2.74$E$-04$&-&$2.18$E$-04$&$-$&$1.61$E$-05$&$-$&$2.28$E$-05$&$-$\\
$ 20$&$1.63$E$-05$&$4.06$&$1.31$E$-05$&$4.06$&$6.94$E$-08$&$7.86$&$9.82$E$-08$&$7.86$\\
$ 40$&$1.07$E$-06$&$3.92$&$8.81$E$-07$&$3.89$&$3.34$E$-10$&$7.69$&$4.73$E$-10$&$7.69$\\
\hline
\multicolumn{9}{|c|}{$\theta = 0.55$}\\
\hline\hline
$ 10$&$4.04$E$-04$&-&$2.65$E$-04$&$-$&$1.61$E$-05$&$-$&$2.28$E$-05$&$-$\\
$ 20$&$4.99$E$-05$&$3.01$&$3.22$E$-05$&$3.04$&$6.96$E$-08$&$7.85$&$9.85$E$-08$&$7.85$\\
$ 40$&$4.72$E$-06$&$3.40$&$2.97$E$-06$&$3.43$&$3.39$E$-10$&$7.68$&$4.80$E$-10$&$7.68$\\
\hline
\end{tabular}\label{tab:k=3}
\vspace{-10pt}
\end{center}
\end{table}


%
\vspace{-10pt}\begin{figure}
\caption{Discretization errors for DG solution to
equation
(\ref{eqn:testeqn})
with $k = 2$.
\label{fig:DiscErrk2}}
\vspace{-10pt}
\begin{minipage}[t]{0.5\textwidth}
\begin{subfigure}{\textwidth}
\includegraphics[width=\textwidth]{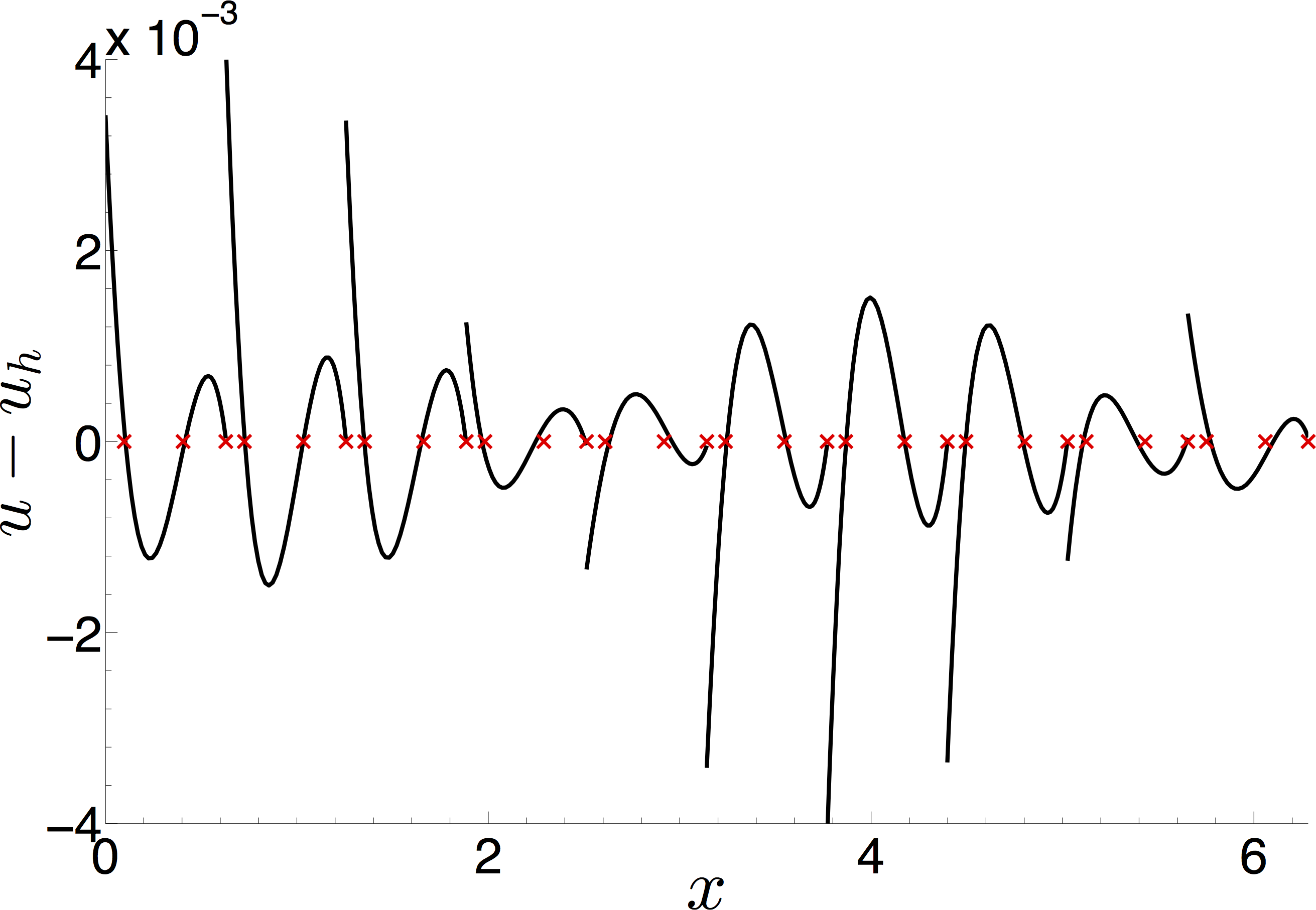}
  \caption{$\theta=1$}
  \end{subfigure}
\end{minipage}
\begin{minipage}[t]{0.5\textwidth}
\begin{subfigure}{\textwidth}
\includegraphics[width=\textwidth]{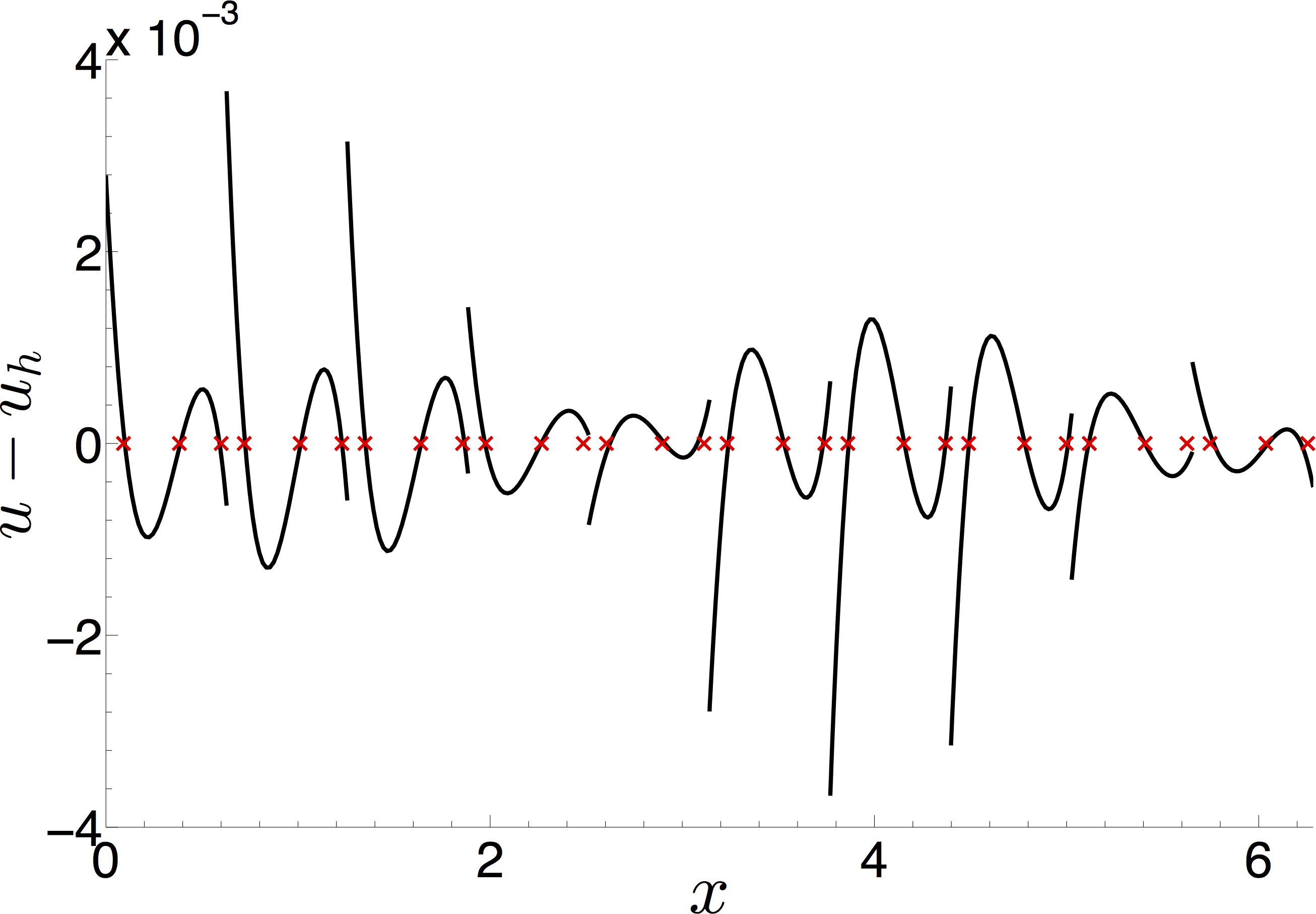}
  \caption{$\theta=0.85$}
    \end{subfigure}
\end{minipage}
\\
\begin{minipage}[t]{0.5\textwidth}
\begin{subfigure}{\textwidth}
\includegraphics[width=\textwidth]{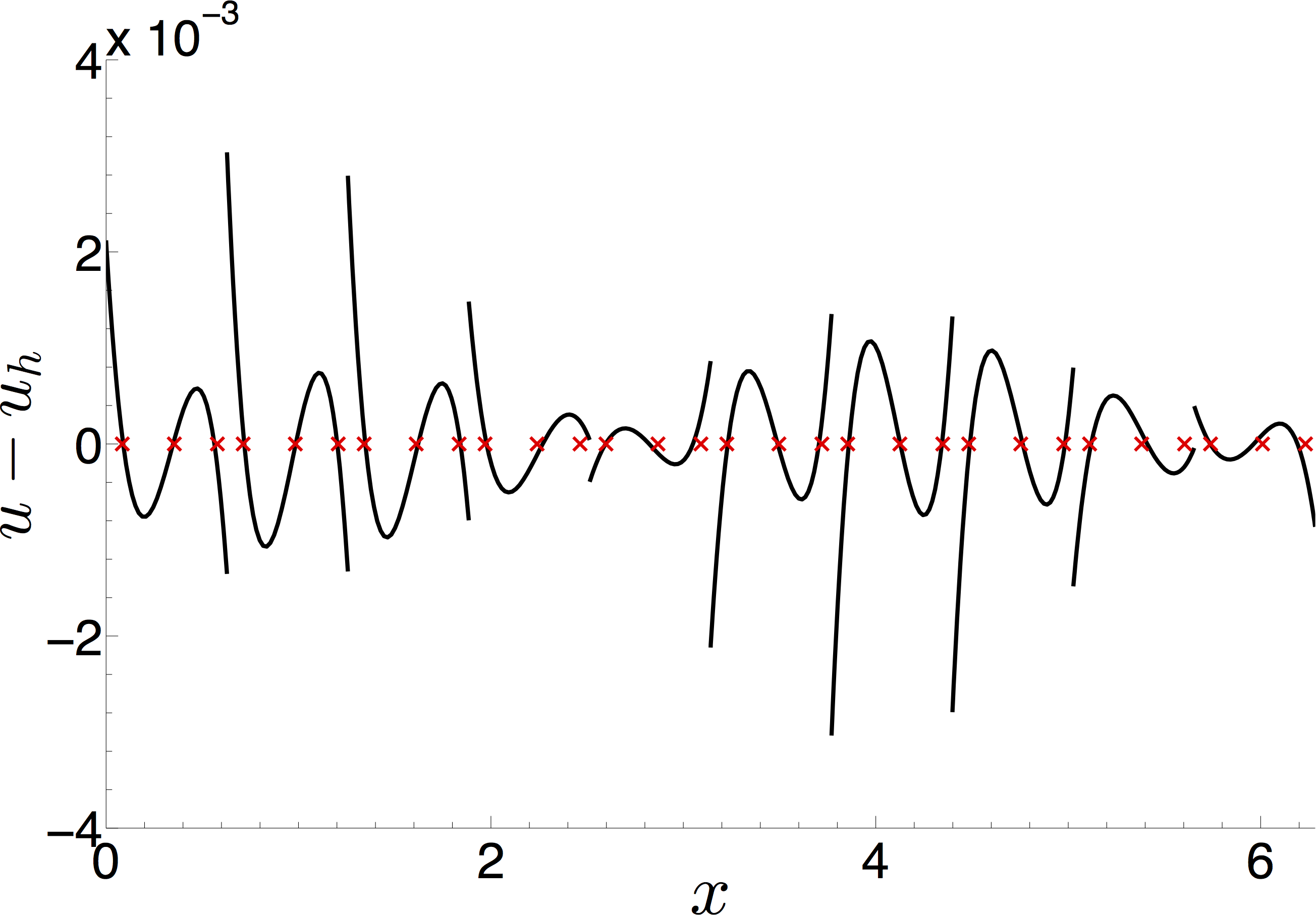}
  \caption{$\theta=0.7$}
    \end{subfigure}
\end{minipage}
\begin{minipage}[t]{0.5\textwidth}
\begin{subfigure}{\textwidth}
\includegraphics[width=\textwidth]{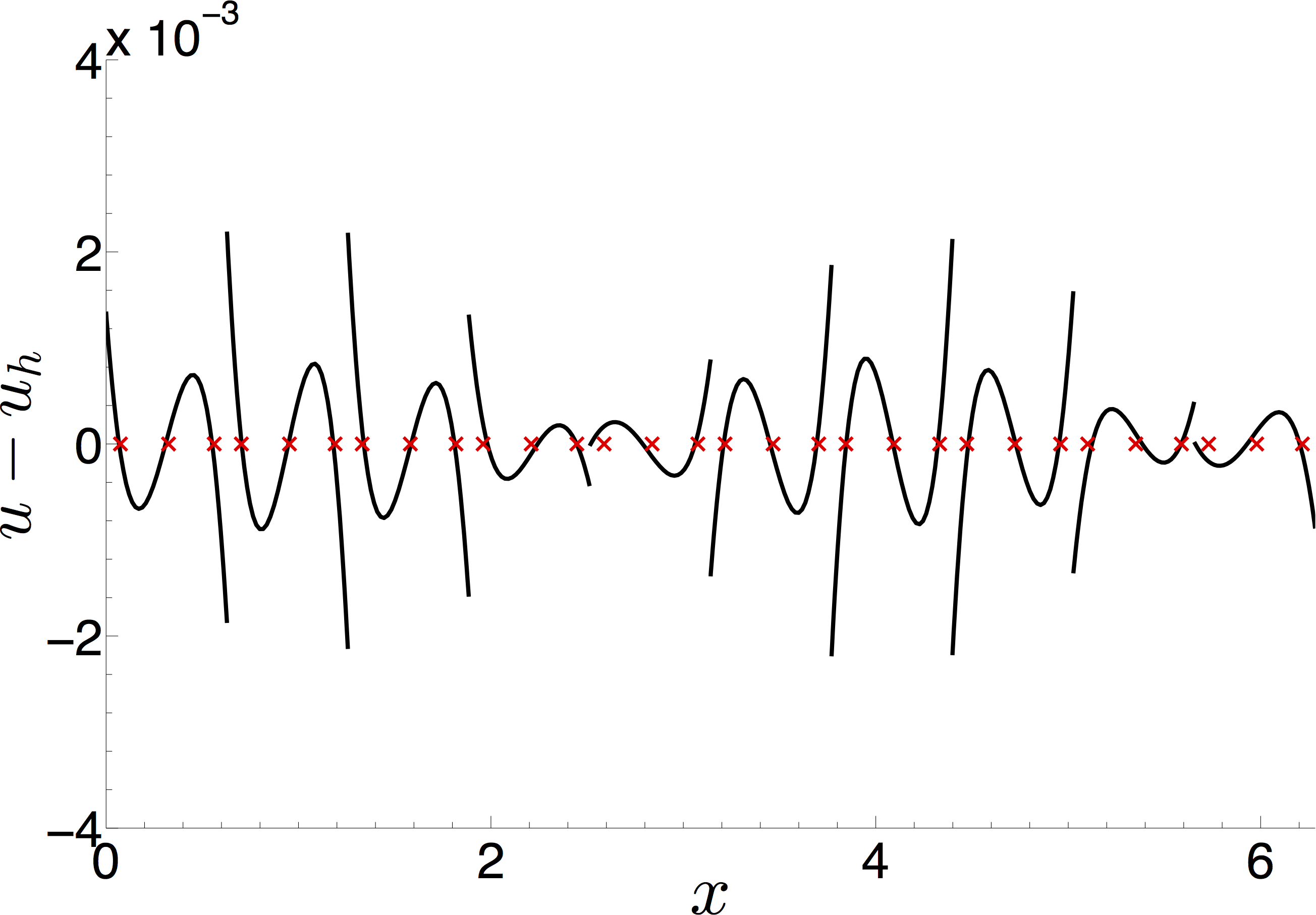}
  \caption{$\theta=0.55$}
    \end{subfigure}
\end{minipage}
\vspace{-10pt}
\end{figure}
\vspace{-10pt}
\vspace{-10pt}
\begin{figure}
\caption{DG and filtered errors for $k = 2$ at time $T = 1$.\label{fig:Errorsk2}}
\vspace{-10pt}
\begin{minipage}[t]{0.44\textwidth}
\begin{subfigure}{\textwidth}
\includegraphics[width=0.95\textwidth]{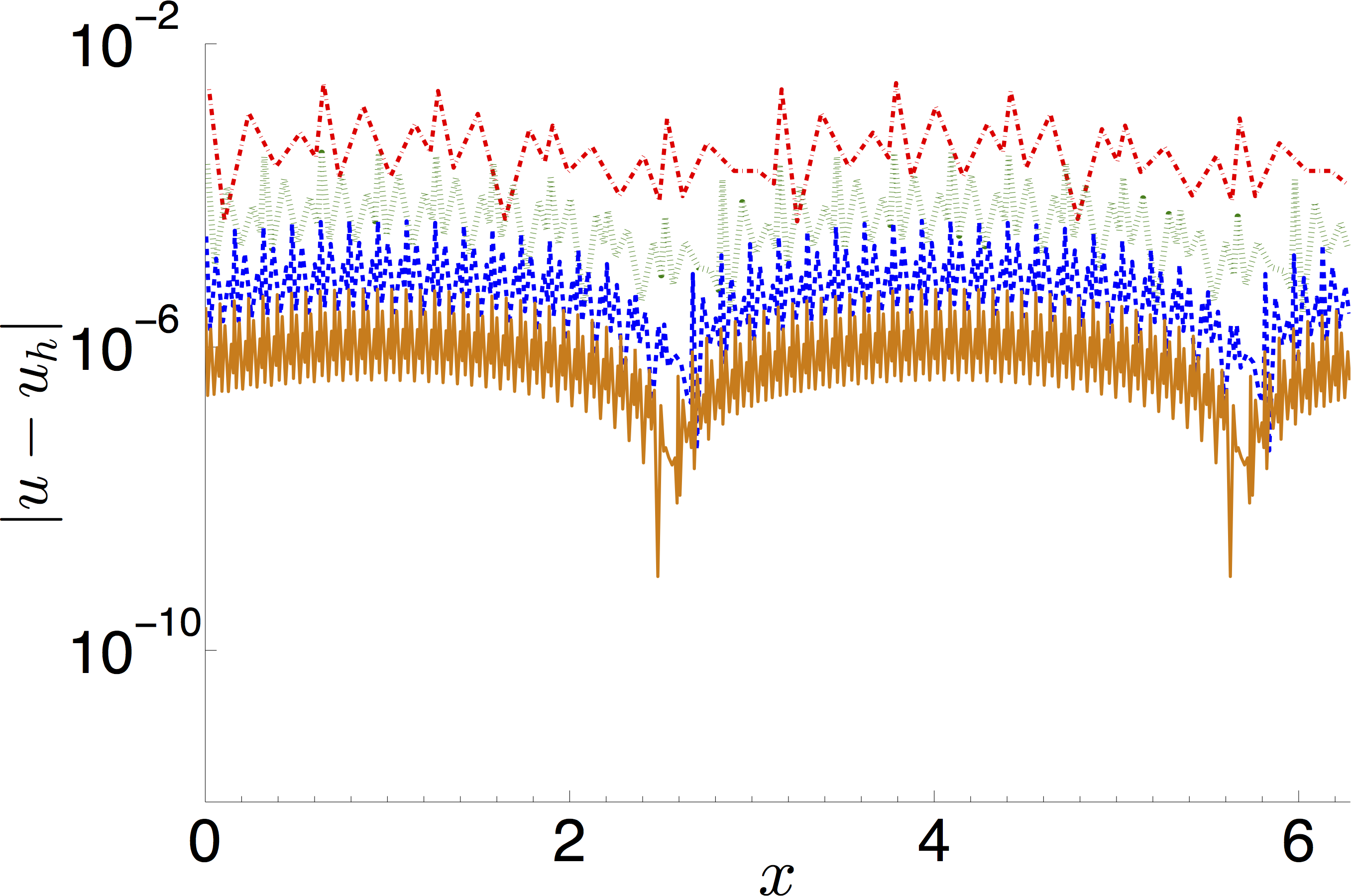}
  \end{subfigure}
\end{minipage}
\begin{minipage}[t]{0.44\textwidth}
\begin{subfigure}{\textwidth}
\includegraphics[width=\textwidth]{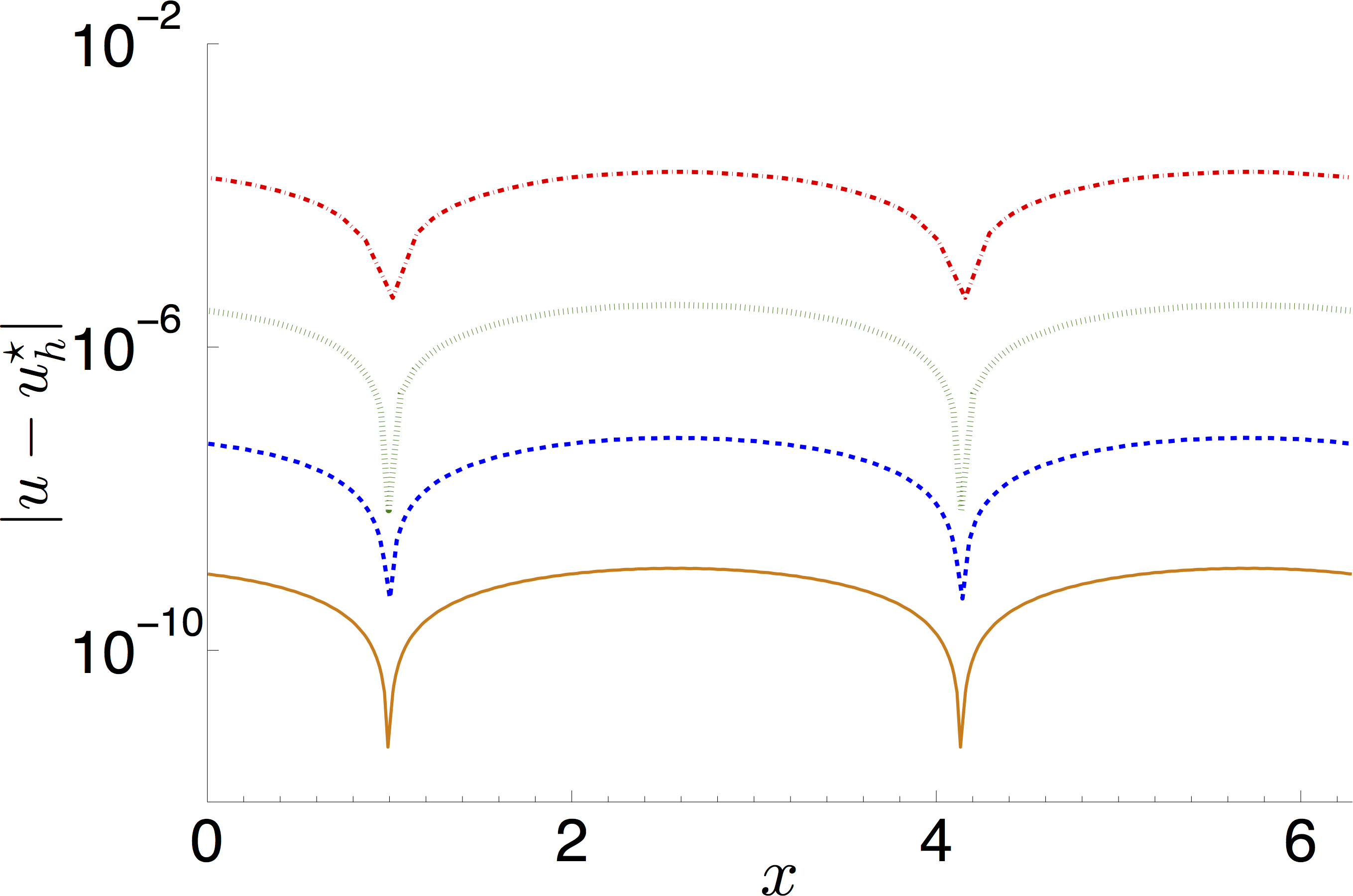}
    \end{subfigure}
\end{minipage}
\begin{minipage}[t]{0.1\textwidth}
\begin{subfigure}{\textwidth}
\vspace*{-0.4cm}
\includegraphics[width=\textwidth]{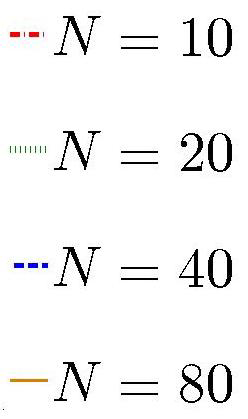}
    \end{subfigure}
\end{minipage}
\begin{center}
(a) Before and after post-processing for $\theta = 1$.
\end{center}
\begin{minipage}[t]{0.44\textwidth}
\begin{subfigure}{\textwidth}
\includegraphics[width=0.95\textwidth]{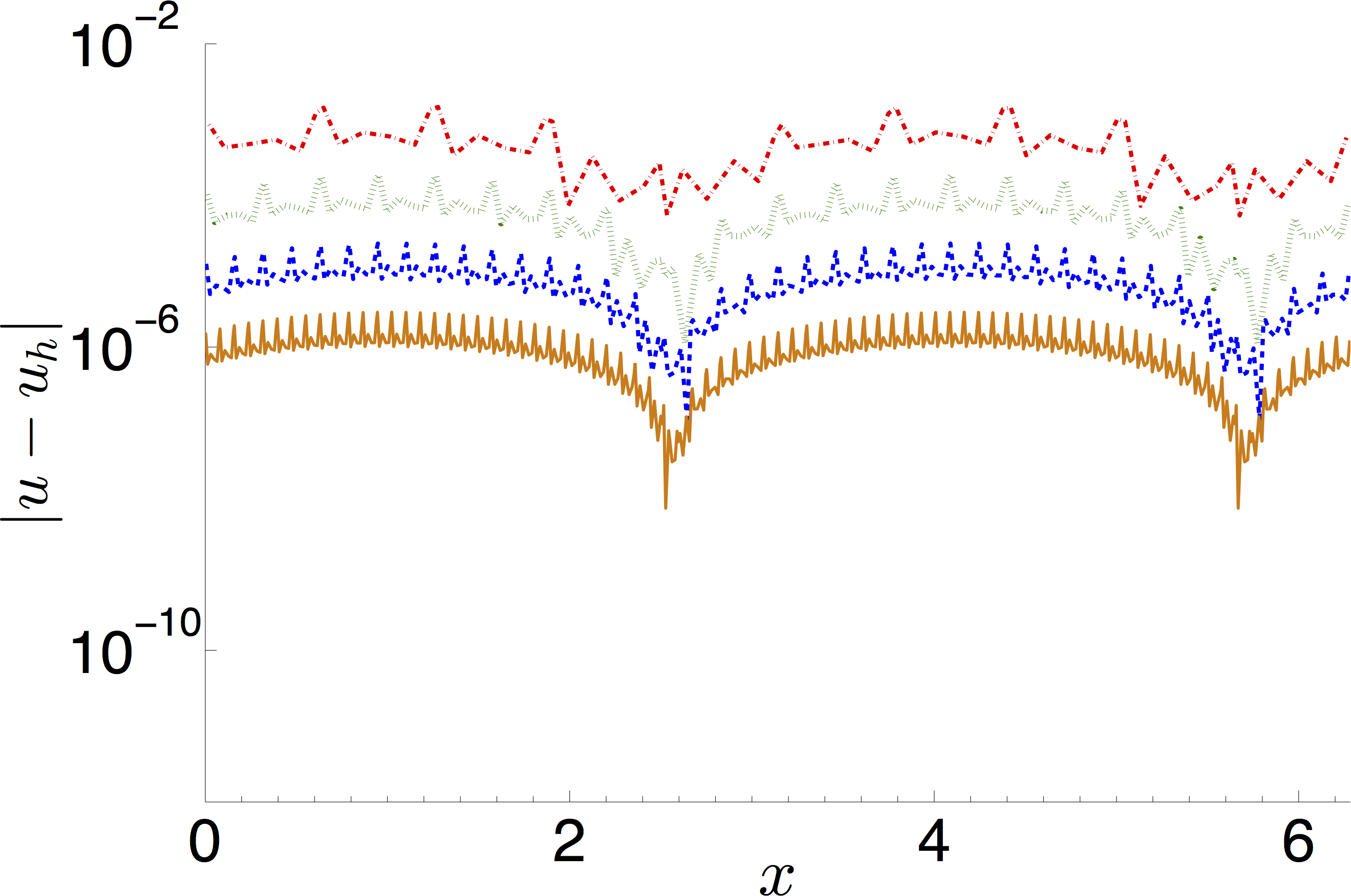}
  \end{subfigure}
\end{minipage}
\begin{minipage}[t]{0.44\textwidth}
\begin{subfigure}{\textwidth}
\includegraphics[width=\textwidth]{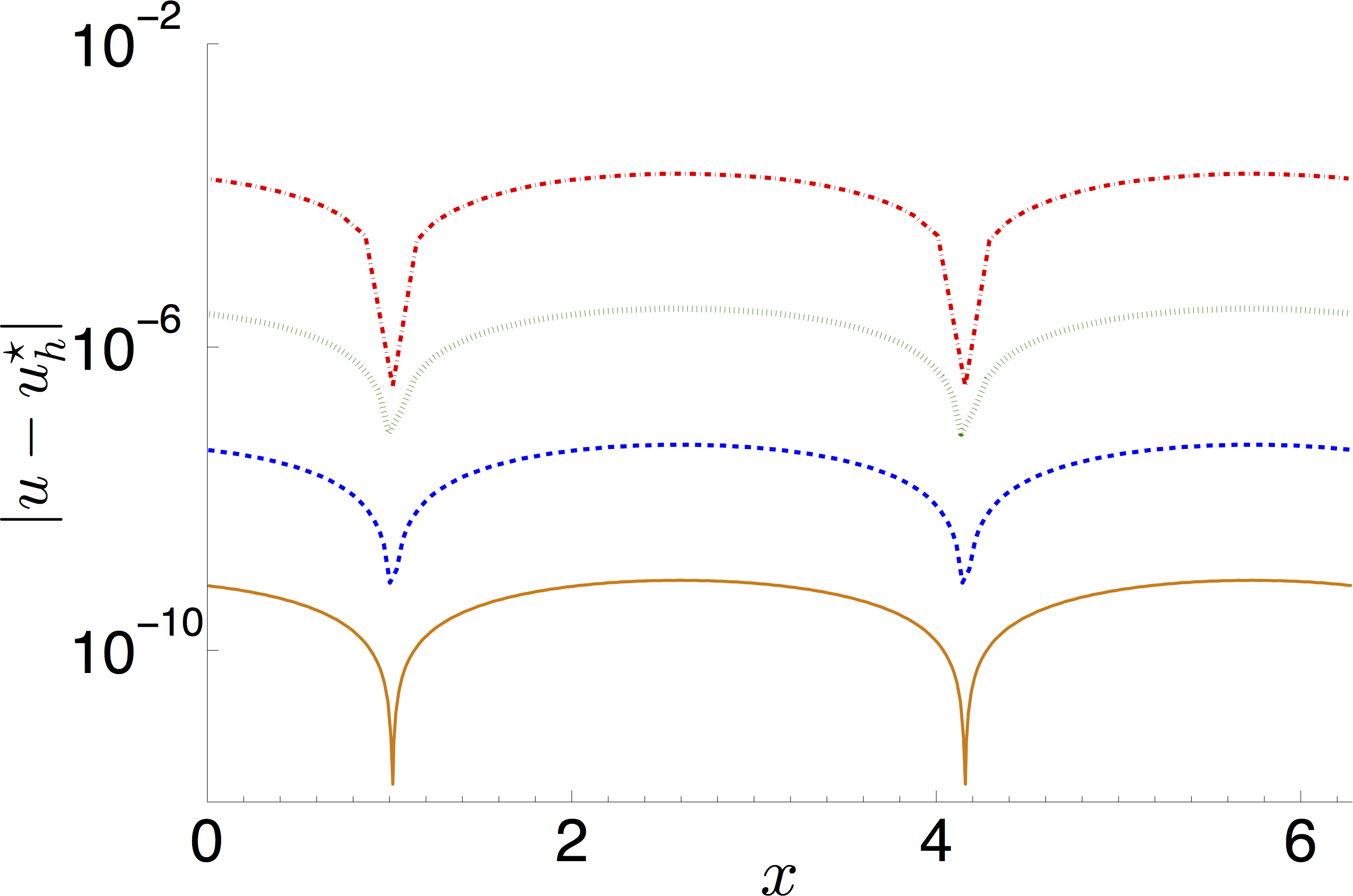}
    \end{subfigure}
\end{minipage}
\begin{minipage}[t]{0.1\textwidth}
\begin{subfigure}{\textwidth}
\vspace*{-0.4cm}
\includegraphics[width=\textwidth]{Legend.png}
    \end{subfigure}
\end{minipage}
\begin{center}
(c) Before and after post-processing for $\theta = 0.55$.
\end{center}
\vspace{-10pt}
\end{figure}

%

\begin{figure}
\caption{Discretization errors for DG solution to
equation
(\ref{eqn:testeqn})
with $k = 3$.
\label{fig:DiscErrk3}}\vspace{-10pt}
\begin{minipage}[t]{0.5\textwidth}
\begin{subfigure}{\textwidth}
\includegraphics[width=\textwidth]{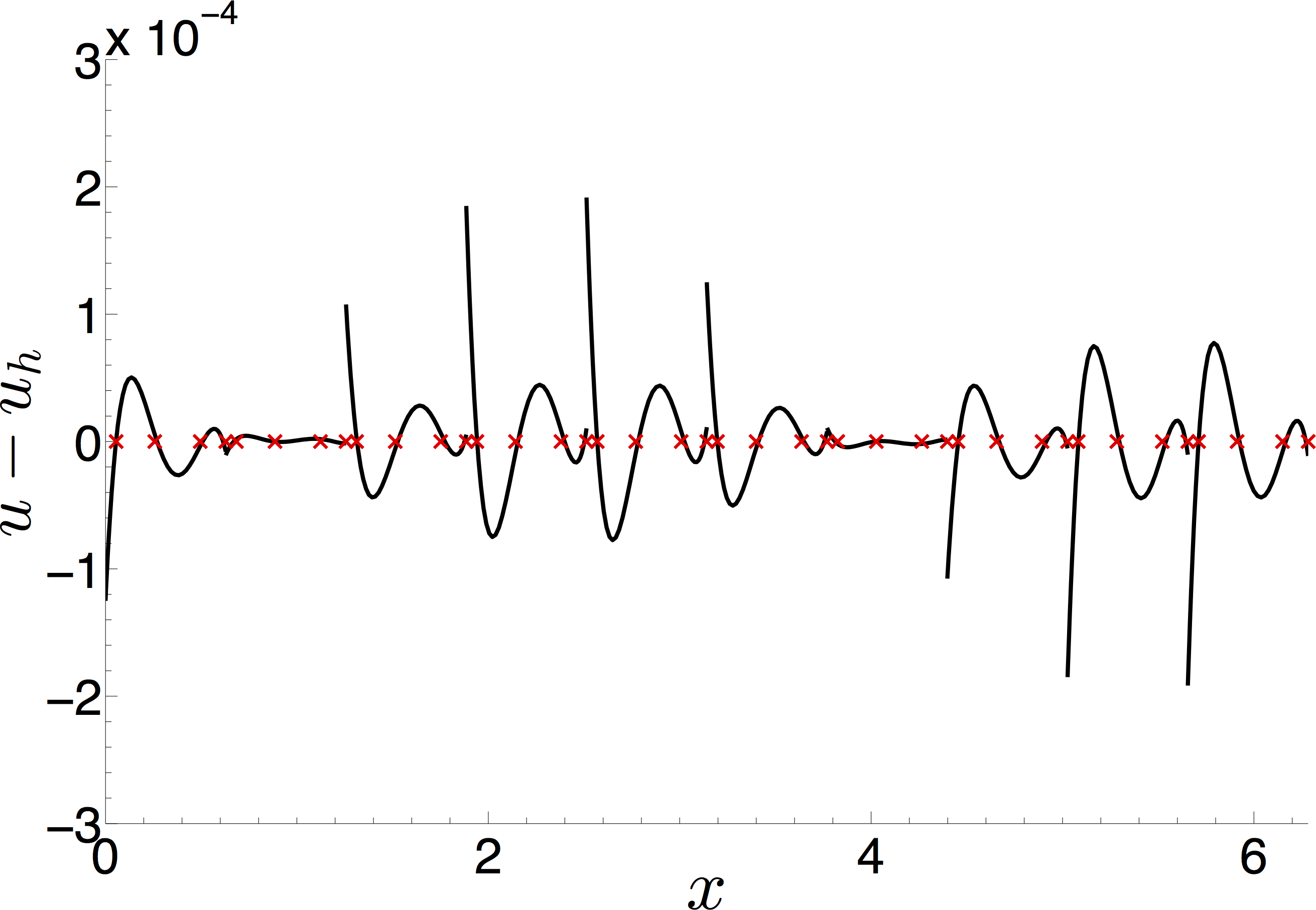}
  \caption{$\theta=1$}
  \end{subfigure}
\end{minipage}
\begin{minipage}[t]{0.5\textwidth}
\begin{subfigure}{\textwidth}
\includegraphics[width=\textwidth]{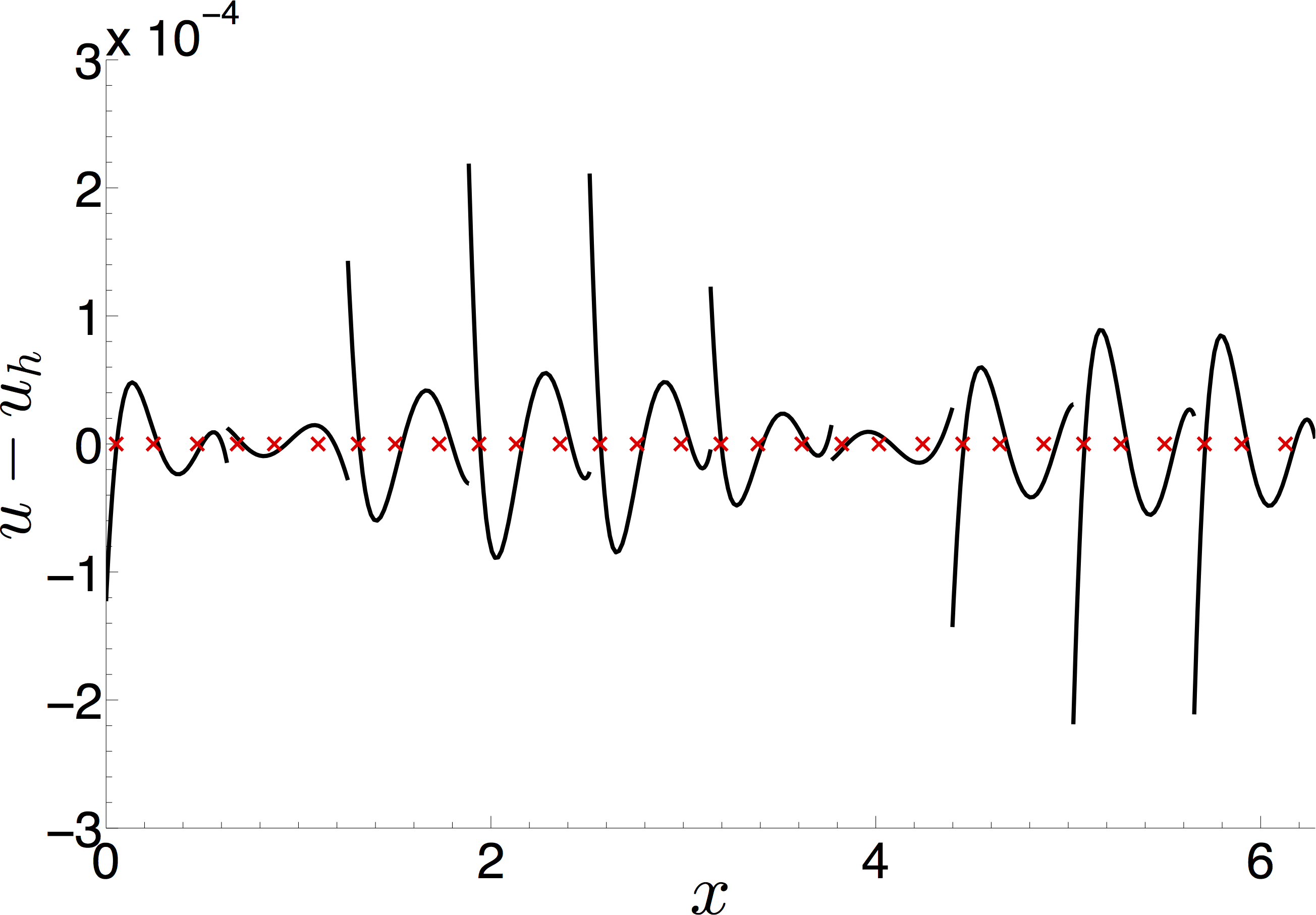}
  \caption{$\theta=0.85$}
    \end{subfigure}
\end{minipage}
\\
\begin{minipage}[t]{0.5\textwidth}
\begin{subfigure}{\textwidth}
\includegraphics[width=\textwidth]{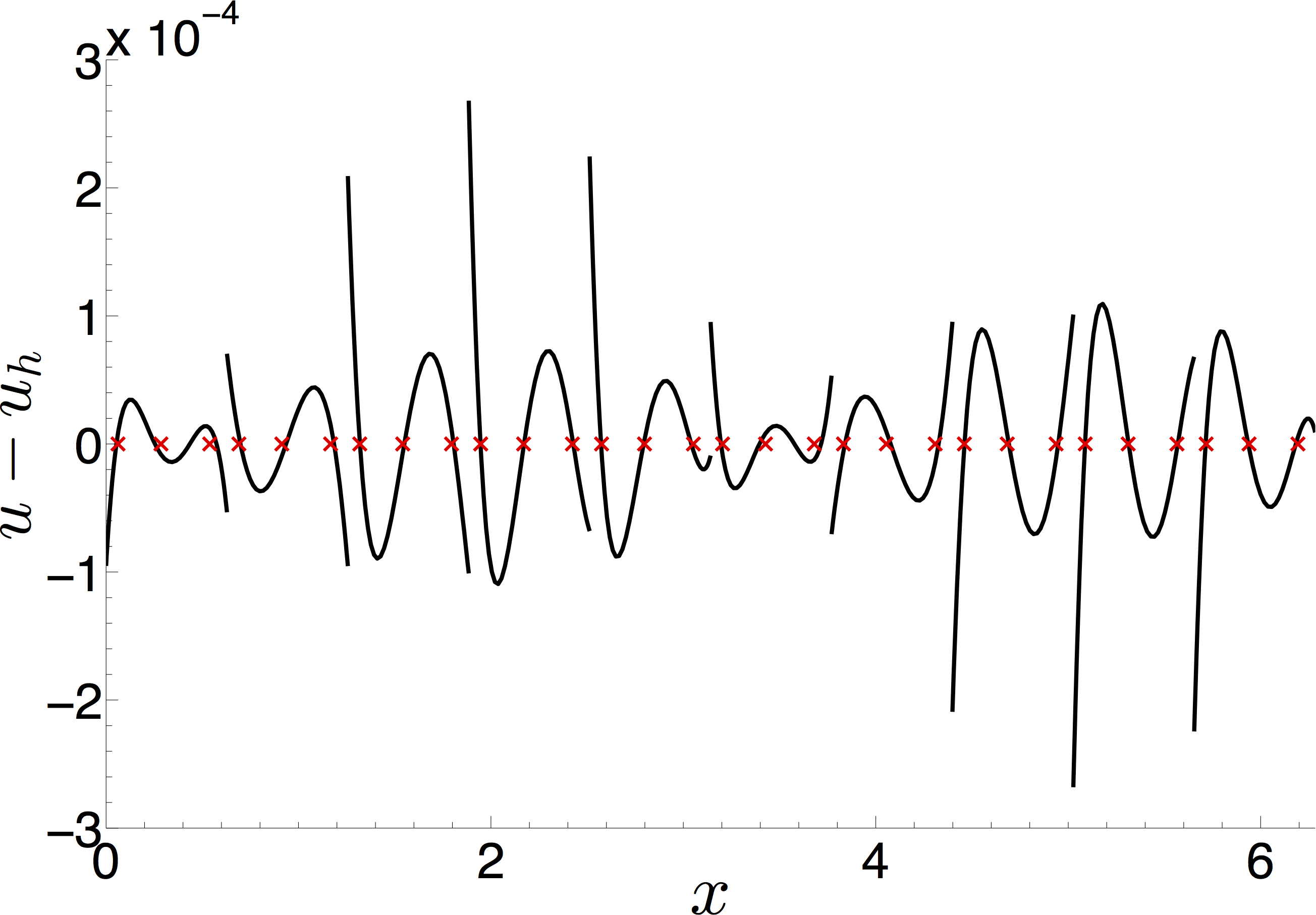}
  \caption{$\theta=0.7$}
    \end{subfigure}
\end{minipage}
\begin{minipage}[t]{0.5\textwidth}
\begin{subfigure}{\textwidth}
\includegraphics[width=\textwidth]{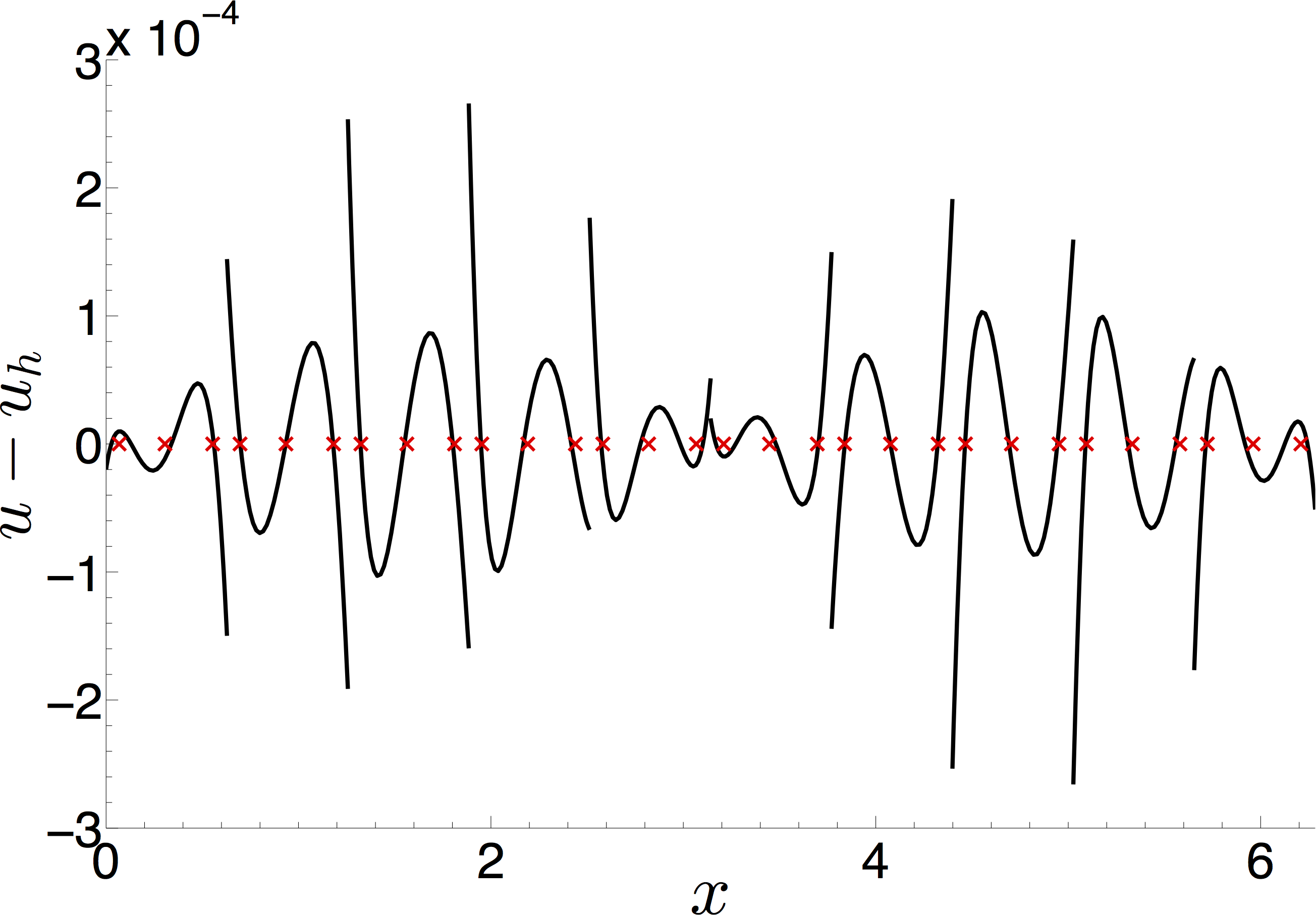}
  \caption{$\theta=0.55$}
    \end{subfigure}
\end{minipage}\vspace{-10pt}
\end{figure}


%

\begin{figure}
\caption{DG and filtered errors for $k = 3$ at time $T = 1$.\label{fig:Errorsk3}}\vspace{-10pt}
\begin{minipage}[t]{0.44\textwidth}
\begin{subfigure}{\textwidth}
\includegraphics[width=0.95\textwidth]{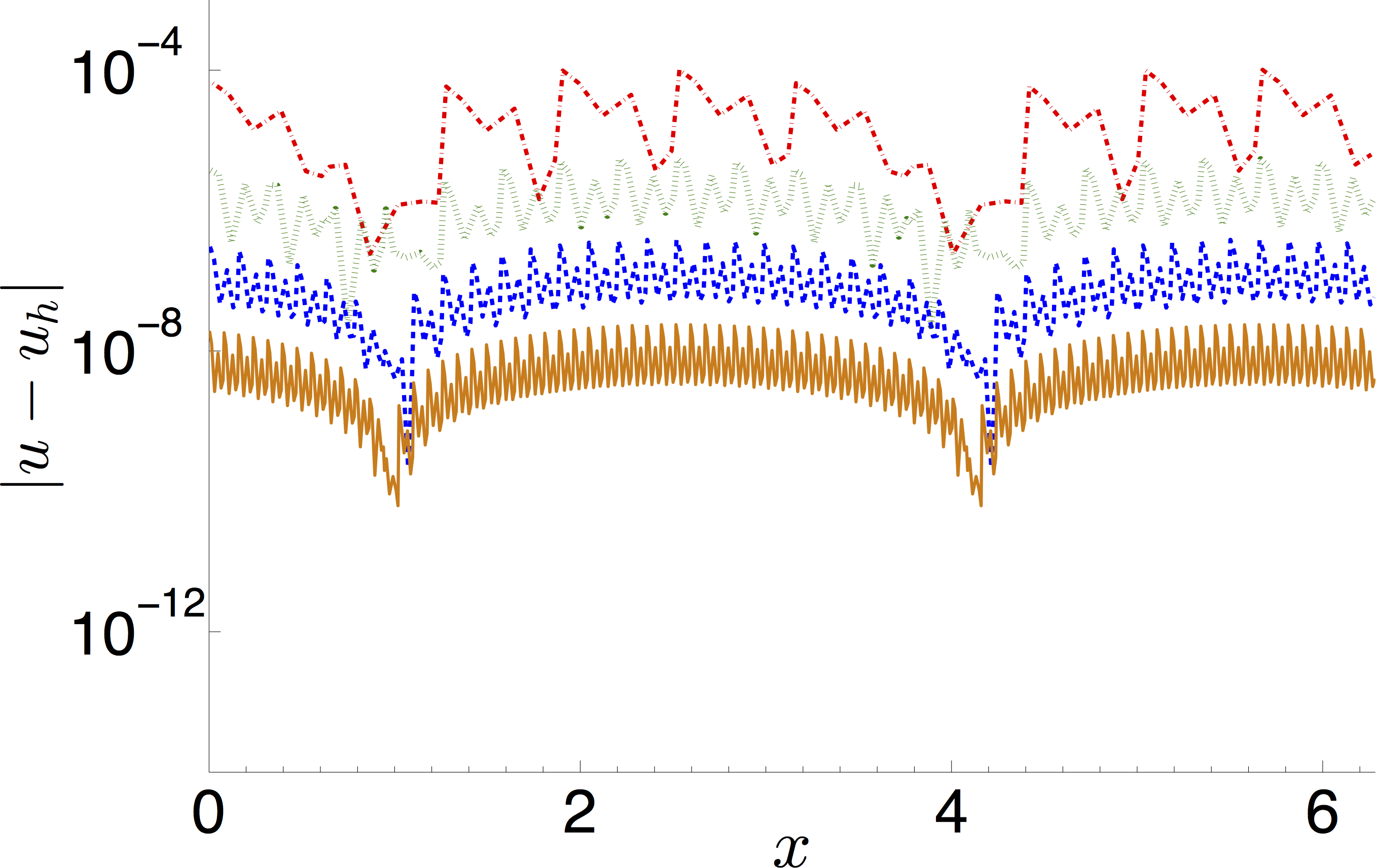}
  \end{subfigure}
\end{minipage}
\begin{minipage}[t]{0.44\textwidth}
\begin{subfigure}{\textwidth}
\includegraphics[width=\textwidth]{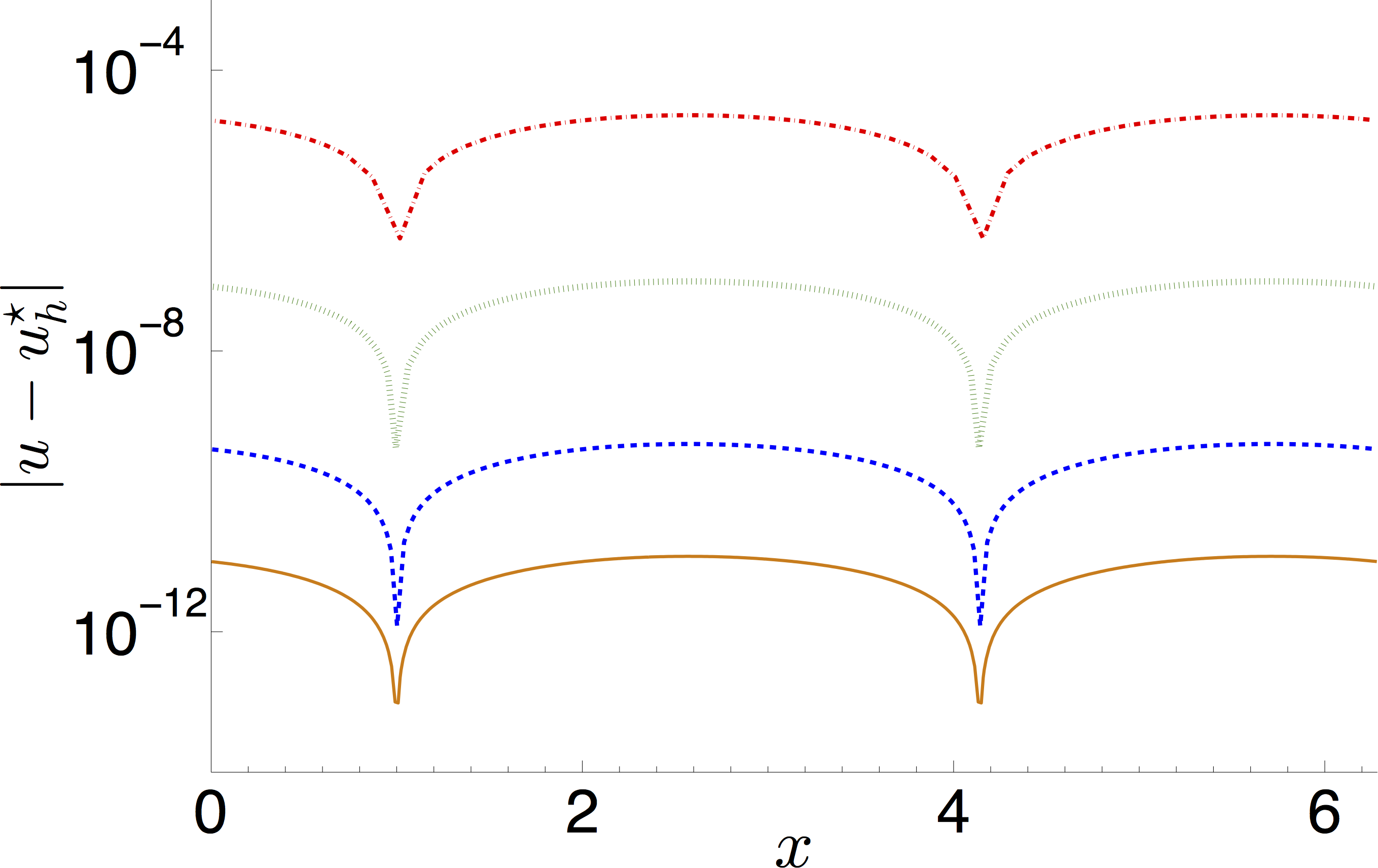}
    \end{subfigure}
\end{minipage}
\begin{minipage}[t]{0.1\textwidth}
\begin{subfigure}{\textwidth}
\vspace*{-0.5cm}
\includegraphics[width=\textwidth]{Legend.png}
    \end{subfigure}
\end{minipage}
\begin{center}
(a) Before and after post-processing for $\theta = 1$.
\end{center}
\begin{minipage}[t]{0.44\textwidth}
\begin{subfigure}{\textwidth}
\includegraphics[width=0.95\textwidth]{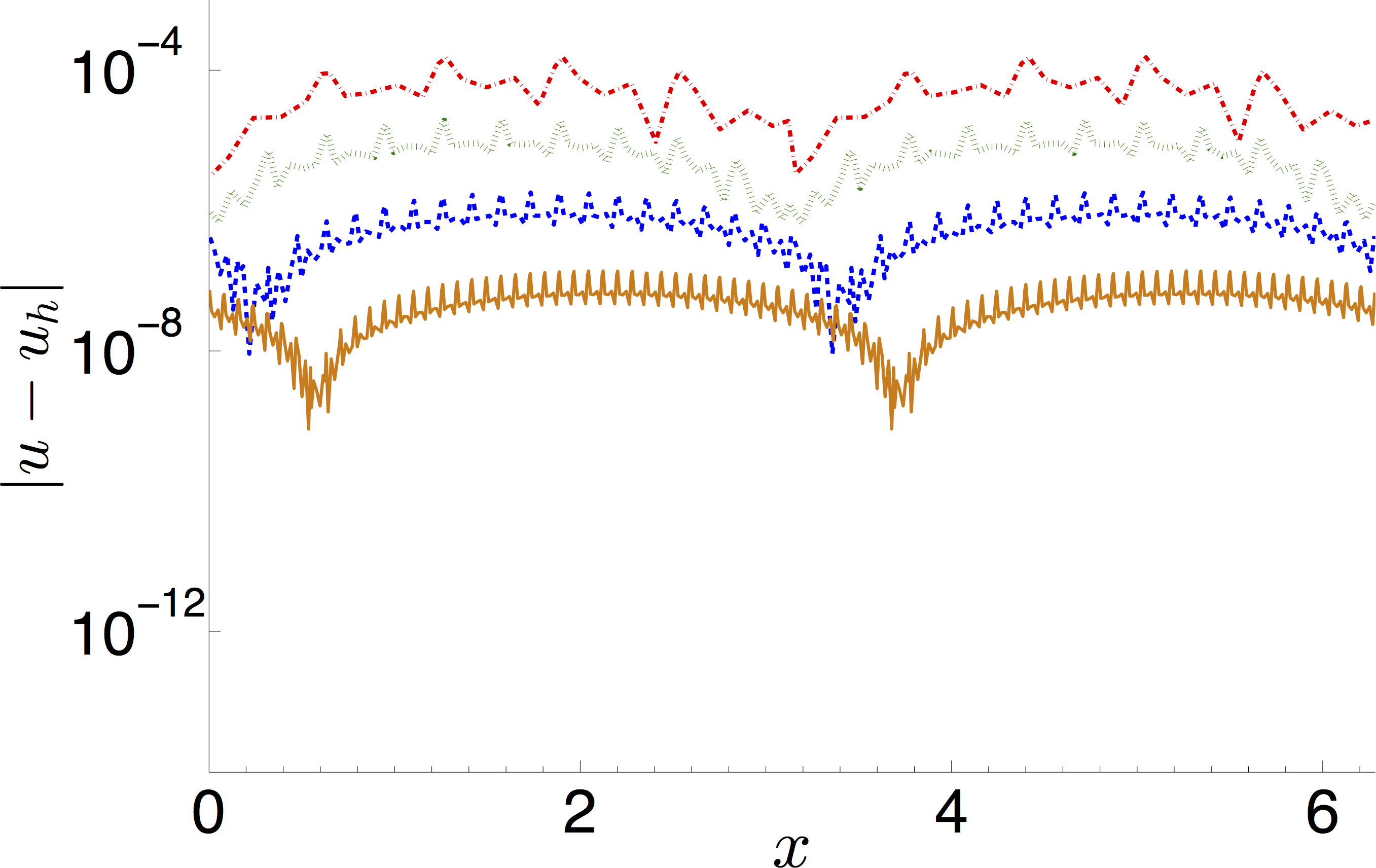}
  \end{subfigure}
\end{minipage}
\begin{minipage}[t]{0.44\textwidth}
\begin{subfigure}{\textwidth}
\includegraphics[width=\textwidth]{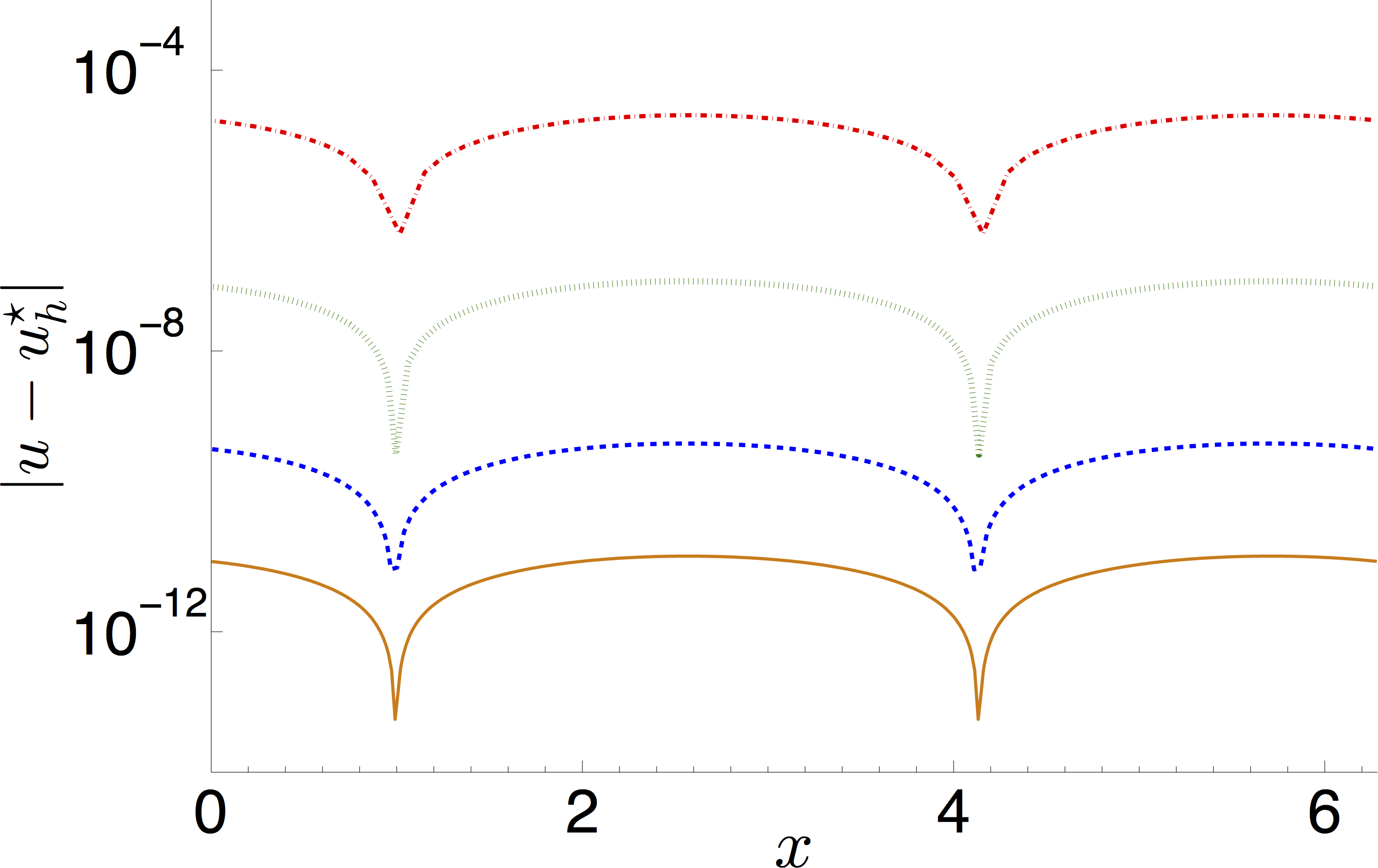}
    \end{subfigure}
\end{minipage}
\begin{minipage}[t]{0.1\textwidth}
\begin{subfigure}{\textwidth}
\vspace*{-0.5cm}
\includegraphics[width=\textwidth]{Legend.png}
    \end{subfigure}
\end{minipage}
\begin{center}
(c) Before and after post-processing for $\theta = 0.55$.
\end{center}\vspace{-10pt}
\end{figure}

%
%

%
%
\section{Conclusions}
\label{sec:conclude}

This paper has presented a study of how the superconvergent pointwise error estimates and negative-order norm analysis changes with a change in flux.  We have proven that for even-degree polynomials, the method is locally superconvergent of order $k+2$ at a linear combination of the right- and left-Radau polynomials.  This also holds true for odd degree polynomials provided a suitable global initial projection is defined.  Further, we have demonstrated that only the constant in the negative-order norm analysis changes for the upwind-biased flux and hence we are able to extract a superconvergent solution of $\mathcal{O}(h^{2k+1}).$    We believe that these estimate will be useful for designing optimal fully discrete schemes that take advantage of this superconvergence information.  We leave this as well as the analysis of non-linear equations for future work.
\\
\\\footnotesize
\textbf{Acknowledgements}\quad
The authors' research is supported by the Air Force Office of Scientific Research (AFOSR), Air Force Material Command, USAF, under grant number FA8655-13-1-3017. We would also like to thank Xinghui Zhong and Jingmei Qiu for guidance in the dispersion analysis and Xiong Meng for many useful discussions. 
%
%
\nocite{*}
\FloatBarrier
\bibliographystyle{plain}      

\end{document}